\documentclass[12pt,oneside]{amsart}
\usepackage[letterpaper]{geometry}
\usepackage[colorlinks, allcolors=blue]{hyperref}
\usepackage{amsmath,amsthm,amssymb,tikz,ifthen,enumitem,placeins}
\usepackage{soul}

\allowdisplaybreaks 

\usepackage{thmtools} 
\usepackage{cleveref} 
\newtheorem{thm}{Theorem}
\numberwithin{thm}{section}
\newtheorem*{thm*}{Theorem}

\newtheorem{prop}[thm]{Proposition}

\newtheorem*{conjecture*}{Conjecture}
\theoremstyle{definition}
\newtheorem{defn}[thm]{Definition}
\newtheorem{remark}[thm]{Remark}
\newtheorem*{remark*}{Remark}

\usepackage[labelfont=sc, justification=centering, textformat=period]{caption}
\makeatletter \g@addto@macro\@floatboxreset\centering \makeatother 

\newcommand\abs[1]{\left\vert#1\right\vert}
\newcommand\ang[1]{\left\langle#1\right\rangle}

\newcommand\floor[1]{{\left\lfloor#1\right\rfloor}}

\let\tilde\widetilde
\renewcommand\bar[1]{\overline{#1}}
\renewcommand\dots{...} 

\renewcommand\Im{\operatorname{Im}}
\providecommand\id{\mathrm{id}}
\DeclareMathOperator\Cl{Cl}
\DeclareMathOperator\Int{Int}

\newcommand\D{{\mathbb D}}

\newcommand\R{{\mathbb R}}
\renewcommand\S{{\mathbb S}}
\newcommand\Z{{\mathbb Z}}
\newcommand\G{{\Gamma}}
\newcommand\g{{\gamma}}
\newcommand\Fc{{\mathcal F}}
\newcommand{\psu}{\mathrm{PSU}(1,1)}
\newcommand\Hc{{\mathcal H}}
\providecommand\Matrix[1]{ \begin{pmatrix} #1 \end{pmatrix} }

\setlist[itemize]{label=\textbullet}

\newcommand\STRING{s_1,s_2, \dots, s_\ell}
\newcommand\Tdef{Teich\-m\"uller defor\-ma\-tion}
\newcommand\A{\mathcal{A}}

\newcounter{commentcounter}\newcommand\COMMENT[2][red]{\stepcounter{commentcounter}\rlap{\smash{$^{\fcolorbox{#1}{#1!15}{\scriptsize\ifthenelse{\equal{#1}{green}}{\color{green!75!black}}{\color{#1}}\!\!{\bf\thecommentcounter}\!\!}}$}}\marginpar{\!\!\parbox{2.8cm}{\raggedright\small \ifthenelse{\equal{#1}{green}}{\color{green!67!black}}{\color{#1}} \textbf{\thecommentcounter.}\,#2}}} \usepackage{silence} 

\title{Reduction theory for {F}uchsian groups~with cusps}
\author{Adam Abrams}
\address{Faculty of Pure and Applied Mathematics, Wroc\l{}aw University of Science and Technology, Wroc\l{}aw, 50370, Poland}
\email{the.adam.abrams@gmail.com}
\author{Svetlana Katok}
\address{Department of Mathematics, The Pennsylvania State University, University Park, PA 16802}
\email{sxk37@psu.edu}
\author{Ilie Ugarcovici}
\address{Department of Mathematical Sciences, DePaul University, Chicago, IL 60614}
\email{iugarcov@depaul.edu}
\date{August 15, 2025}

\begin{document}
\begin{abstract}
We study a family of Bowen--Series-like maps associated to any finitely generated Fuchsian group of the first kind with at least one cusp. These maps act on the boundary of the hyperbolic plane in a piecewise manner by generators of the group. We show that the two-dimensional natural extension (reduction map) of the boundary map has a domain of bijectivity and global attractor with a finite rectangular structure, confirming a conjecture of Don Zagier. Our work is based on the construction of a special fundamental polygon, related to the free product structure of the group, whose marking is preserved by ``Teichm\"uller deformation.''
\end{abstract}
\maketitle

\section{Introduction}\label{sec:intro}
Let~$\G$ be a finitely generated Fuchsian group of the first kind (finite covolume) acting on the hyperbolic plane properly discontinuously by orientation-preserving isometries. In the upper half-plane model $\Hc$, $\G<\mathrm{PSL}(2,\R)$, and in the unit disk model~$\D$, $\G<\mathrm{PSU}(1,1)$. In this paper, we mostly use the disk model~$\D$, leaving the upper half-plane $\Hc$ only for some classical examples.

We denote the Euclidean boundary for either model by~$\S$: for the upper half-plane $\S = \partial\Hc = \R P^1$, and for the unit disk $\S = \partial\D = S^1$ is the circle at infinity. Recall that $\G$ is of the first kind if its limit set is all of $\S$. If $\G$ is a finitely generated Fuchsian group of the first kind, it has finite covolume \cite[Theorems 4.6.1 and 4.5.1]{Katok-book}, $r$ conjugacy classes of maximal elliptic cyclic subgroups of orders $m_1, m_2,\dots, m_r$ (all $m_i\ge 2$, listed in non-decreasing order) and $t$ conjugacy classes of maximal parabolic subgroups. The quotient space $\G\backslash\D$ is a Riemann surface of genus $g\ge 0$ with $r$ ramification points (with ramification numbers $m_1, m_2,\dots, m_r$) and $t$ cusps (to which it is convenient to assign ramification numbers~$\infty$). Such $\G\backslash\D$ are called \emph{Riemann surfaces of finite type} or \emph{Riemann orbifolds}.
The list
\begin{equation}\label{signature}
(g; m_1,m_2,\dots, m_r, \underbrace{\infty,\dots,\infty}_{\text{$t$}}) = (g; m_1,m_2,\dots, m_r;t)
\end{equation}
is called the \emph{signature} of~$\G$. The only restriction on the signature is that
\begin{equation}\label{eq:area}
2g-2+\sum_{i=1}^{r}\left(1-\frac1{m_i}\right)+t>0,
\end{equation}
corresponding to positivity of the hyperbolic area of $\G\backslash\D$~\cite[Theorem 4.3.1]{Katok-book}, which we will always assume. On the other hand, by Poincar\'e's Polygon Theorem, for any signature~\eqref{signature} there exists a finitely generated Fuchsian group of the first kind with this signature.

Throughout the paper, when we refer to a Fuchsian group with a given signature, we always mean a finitely generated Fuchsian group of the first kind.

For such Fuchsian group~$\Gamma$, Bowen and Series~\cite{BowenSeries79} constructed a map---which we refer to as a \emph{boundary map}---acting on~$\S$ in a piecewise manner by generators of~$\G$ which is orbit-equivalent to~$\G$ and Markov with respect to a finite (if~$\G$ has no cusps) or countable (if~$\G$ has cusps) partition of~$\S$. They established some important ergodic properties, which were later used by C.~Series~\cite{Series81} to construct a symbolic representation of the geodesic flow on the Riemann surface~$\Gamma\backslash\mathbb D$. 

Adler and Flatto investigated thoroughly boundary maps in the case of the modular group~\cite{AF84} and surface groups~\cite{AF91}, i.e., co-compact groups with no elliptic elements, using an $(8g-4)$-sided fundamental polygon instead of the $4g$-gon from~\cite{BowenSeries79}. In 2014, Ahmadi and Sanaz~\cite{Ahmadi-Sanaz2014} generalized Adler and Flatto's work on surface groups to co-compact Fuchsian groups with elliptic elements.

In 2005, Don Zagier experimentally observed that for the modular group~$\mathrm{PSL}(2,\Z)$ the two-dimensional map on geodesics (which is the natural extension of the boundary map) has a global attractor with finite rectangular structure. The finite rectangular structure property along with other properties (conjecturally equivalent to it) form, in Zagier's terminology, a ``reduction theory'' for the group. The original motivation of Don Zagier was the inversion problem for modular forms on finitely generated Fuchsian groups with at least one cusp (reconstruction of modular forms from its period cocycles) \cite{Zagier-2023}, where the finite rectangular structure of the attractor is essential. He also conjectured that the rectangular structure persists when the partition points used in defining the boundary map are perturbed in a continuous manner.

In this paper, we prove Zagier's conjecture for all finitely generated Fuchsian groups of the first kind with at least one cusp.\footnote{This class, in particular, contains all congruence subgroups.} 

The main result of this paper is the \hypertarget{main}{following theorem}.

\medskip\noindent\textbf{Main Theorem.} {\it
Given a Fuchsian group with signature $(g; m_1,m_2,\dots, m_r;\mbox{\ensuremath{t\ge 1)}}$, there exists a marked fundamental polygon, a finite partition of \mbox{$\S = \partial\D$}, and a boundary map acting piecewisely by generators of the group with the following properties:
\begin{enumerate}[label=(\roman*)]
    \item \label{main-i} the boundary map has a finite Markov partition,
    \item \label{main-ii} its natural extension has a global attractor with finite rectangular structure,
    \item \label{main-iii} almost every point enters the attractor in finite time.
\end{enumerate}
If, in addition, the group has at least one elliptic element of order greater than~$2$, there exists a continuous family of partitions and associated boundary maps with properties~(ii) and~(iii).
}\smallskip

\Cref{fig:polygon and attractor example} shows an example of a fundamental polygon, partition of the boundary, and the attractor of the natural extension of the associated boundary map.

\begin{figure}[hbt]
    \includegraphics[width=0.95\textwidth]{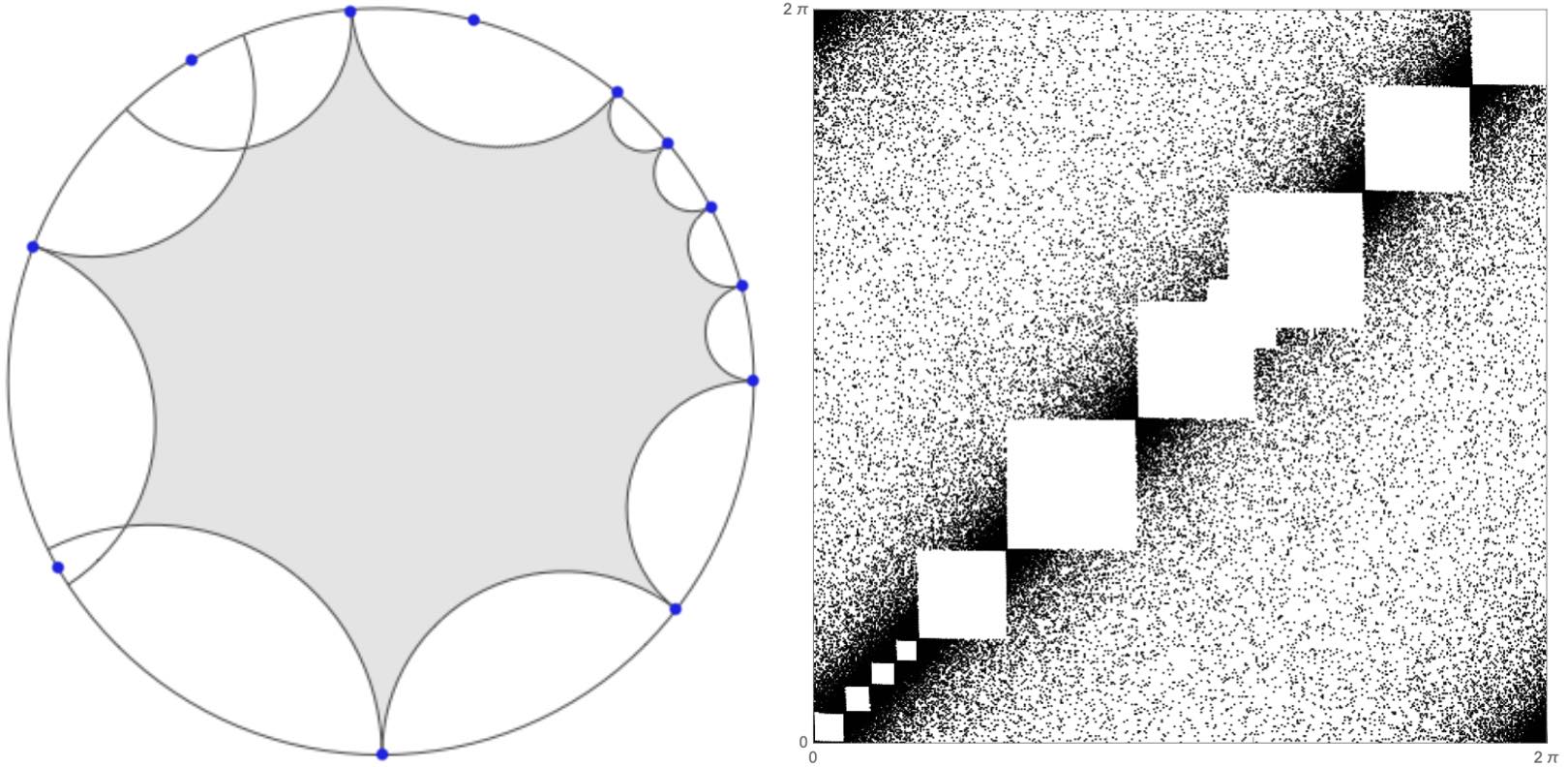}
    \caption{Example of polygon, partition, and attractor (genus~1 with~3 elliptic points and~2 cusps)}
    \label{fig:polygon and attractor example}
\end{figure}

Fuchsian groups with at least one cusp, i.e., with signature $(g,m_1,\dots,m_r;t\ge 1)$, are free products of cyclic groups~\cite{Lehner-1960}, and the generators of these cyclic groups are called \emph{independent generators}.
We consider fundamental polygons related to the free product structure of the group whose side-pairing transformations are independent generators of the group. Such fundamental polygons have many vertices that are ideal (i.e., on the boundary~$\S$), and because of that we call these polygons ``quasi-ideal'' (\Cref{def:qi}).

We prove the main theorem in two steps:
\begin{enumerate}
    \item For a given signature, we construct a ``canonical'' polygon (\Cref{def:canonical}), a canonical Fuchsian group with this signature, and associated boundary maps and their natural extensions. We prove our results for this special situation using its additional features (see Theorems~\ref{thm:bijectivity} and~\ref{thm:attractor}).
    \item The result for a given Fuchsian group is obtained using  the Fenchel--Nielsen map, which conjugates the data for the canonical Fuchsian group (of the same signature) with the data for the given group. See \Cref{sec:Teich} for precise definitions.
\end{enumerate}
Any Fuchsian group with at least one cusp has a quasi-ideal fundamental polygon obtained by ``\Tdef{}'' of the canonical polygon for the same signature (\Cref{thm:FG-qip}) which preserves the combinatorial structure and the marking of the polygon. This allows us to go from step~(1) to step~(2) and also gives an alternative definition of Techm\"uller space as the space of quasi-ideal polygons, up to orientation-preserving isometries of~$\D$ (\Cref{thm:teich-homeo}). The only other case of ``Teichm\"uller friendly'' polygon construction we are aware of is the $(8g-4)$-sided fundamental polygon for surface groups of genus~$g$ in~\cite{AF91, KU-structure}.\footnote{For other constructions of Teichm\"uller space for surface groups via fundamental polygons, see~\cite{ZVC1980, SchmutzSchaller-1999}.} Our construction is notably distinct in that it uses the ideal vertices of fundamental polygons.

\medskip
The paper is organized as follows. In \Cref{sec:groups} we review classical results about fundamental domains and free products, introduce a notion of a canonical polygon, and prove (\Cref{thm:cqi}) that for any signature and any marked canonical polygon, the group generated by its side-pairing transformation is a Fuchsian group with this signature.
In \Cref{sec:Teich} we recall the definition of Teichm\"uller space, describe the \Tdef{} of the canonical polygon, and prove (\Cref{thm:FG-qip}) that any finitely generated Fuchsian group of the first kind has a fundamental polygon obtained by this method. We conclude the section by describing the Teichm\"uller space of a Riemann orbifold as the space of quasi-ideal polygons (\Cref{thm:teich-homeo}).
In \Cref{sec:boundary map} we define boundary maps  and their natural extensions related to quasi-ideal fundamental polygons and prove the finite orbit/cycle property for the boundary maps.
In \Cref{sec:canonical proof} we prove that in the canonical case the two-dimensional natural extension of the boundary map has a domain of bijectivity and a global attractor with finite rectangular structure 
(\Cref{thm:bijectivity} and \Cref{thm:attractor}).
In \Cref{sec:main proof} we complete the proof of the main theorem for any Fuchsian group with at least one cusp, using the \Tdef{}.
The explicit construction of the canonical polygon is contained in Appendix~\ref{sec:explicit}.

\subsection*{Acknowledgments} We would like to thank Ser Peow Tan for insightful remarks about quasi-ideal polygons. We also thank Federico Rodriguez-Hertz and Sergei Tabachnikov for helpful discussions and the anonymous referee for a thorough and careful reading of the paper and useful remarks.

\section{Fuchsian groups and fundamental domains}\label{sec:groups}

\subsection{Classical fundamental domains}
The most common classical fundamental domains for Fuchsian groups are Dirichlet polygons (centered at a point not fixed by any element of the group), which go back to Fricke--Klein, and Ford fundamental polygons~\cite{Ford-1925} obtained using isometric circles of elements in the group. The group with a given signature in Poincar\'e's Polygon Theorem is generated by M\"obius transformations pairing the sides of a fundamental polygon (see~\cite{Jones-Singerman-1987} or~\cite[Theorem 4.3.2]{Katok-book} for the construction and~\cite{Maskit1971} for the proof), which has the following features:
\begin{itemize} \raggedright
    \item the number of sides for a group with signature $(g; m_1,m_2,\dots, m_r;t)$ is $4g+2r+2t$,
    \item each elliptic and parabolic cycle consists of a single vertex,
    \item there is one accidental cycle which represents an ordinary point.
\end{itemize}
Such a fundamental polygon is called a \emph{canonical Poincar\'e polygon} or a \emph{Fricke polygon}, and
the obtained Fuchsian group has the canonical presentation

\allowdisplaybreaks
\begin{equation}
\begin{aligned}\label{presentation}
\G=\bigg\langle &a_1,b_1,\dots\,,a_g,b_g, x_1,\dots,x_r,p_1,\dots \,, p_t:\\&\qquad x_1^{m_1}=\cdots =x_r^{m_r}
=1, \;\prod_{j=1}^t p_j\prod_{i=1}^r x_i\prod_{k=1}^g b_k^{-1}a_k^{-1}b_ka_k = 1\bigg\rangle,
\end{aligned}
\end{equation}
where $a_k$ and $b_k$ are hyperbolic elements (note $b_k^{-1}a_k^{-1}b_ka_k = [a_k,b_k]$ is the commutator right-to-left, consistent with composing transformations), $x_k$ are elliptic, and $p_k$ are parabolic. Conversely, any finitely generated Fuchsian group of the first kind possesses a Fricke polygon (obtained by allowable modifications of a Dirichlet polygon). This result goes back to Klein~\cite{Klein-1923}, see also Lehner~\cite{Lehner-book} and Keen~\cite{Keen-1965}.

\subsection{Free product of cyclic groups}\label{sec:free product} 
If the group has at least one parabolic element, say,~$p_1$, then it can be eliminated from the presentation~\eqref{presentation} and we obtain a group isomorphic to~$\G$ whose generators are those of~$\G$ with~$p_1$ deleted, and whose  only relations are elliptic. In other words, such a group is a free product of~$r$ cyclic subgroups of finite order
and $2g+t-1$ cyclic subgroups of infinite order, of which~$2g$ generators are hyperbolic and are related to the genus, and $t-1$ generators are parabolic fixing non-congruent cusps. The set of generators of these cyclic groups is called an \emph{independent set of generators} for~$\G$. More precisely, the following theorem holds:

\begin{thm*}[Lehner~\cite{Lehner-1960}] Let~$\G$ be a Fuchsian group with signature
$(g;m_1, \dots, m_r;t)$ with $t \ge 1$. Then 
\begin{equation}\label{freeproduct}
\Gamma\cong 
\Z_{m_1} * \Z_{m_2} * \cdots * \Z_{m_r} * \underbrace{ \Z * \cdots * \Z }_{2g+t-1}.
\end{equation}
\end{thm*}

The idea of using the free product of groups in order to provide examples for a great variety of Fuchsian groups by the ``method of free combination'' goes back to Klein~\cite{Fricke-Klein-1897} and consequently appears in the 1929 Ford book~\cite{Ford-1951}.
See also Lehner~\cite{Lehner-1960} and Tukia~\cite{Tukia1972}. The classical examples of this construction are {Schottky groups}. Using this technique,
Lehner~\cite{Lehner-1960} constructed subgroups of the modular group with arbitrary signature.

\begin{remark}
Notice that Dirichlet domains and Fricke's polygons for congruence subgroups are not, in general, related to independent sets of generators. For example, the classical fundamental polygon for the modular group is not; an alternative fundamental polygon using $\mathrm{PSL}(2,\Z) \cong \Z_2 * \Z_3$ is shown on the right of \Cref{fig:two modular polygons}.

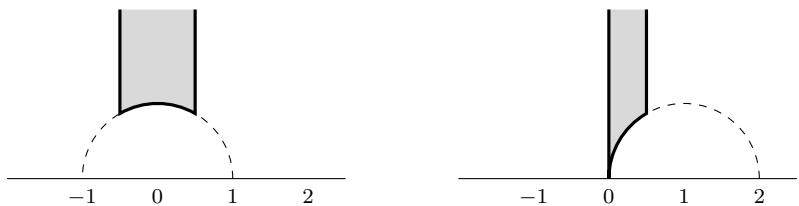
\begin{figure}[hbt]
    \begin{tikzpicture}[scale=1]
        \def\Inf{2.25}
        \draw (-2,0) -- (2.5,0);
        \foreach \x in {-1,0,1,2} \draw (\x,0) node [below] {\tiny$\x$};
        \draw [dashed] (1,0) arc (0:180:1);
        \draw [very thick, fill = black!15] (-1/2,\Inf) -- (120:1) arc (120:60:1) -- (1/2,\Inf);

        \begin{scope}[xshift=6cm]
        \draw (-2,0) -- (2.5,0);
        \foreach \x in {-1,0,1,2} \draw (\x,0) node [below] {\tiny$\x$};
        \draw [dashed] (2,0) arc (0:180:1);
        \draw [very thick, fill = black!15] (0,\Inf) -- (0,0) arc (180:120:1) -- (1/2,\Inf);
        \end{scope}
    \end{tikzpicture}
    \caption{Dirichlet (left) and quasi-ideal (right) fundamental polygons for the modular group in the half-plane model}
    \label{fig:two modular polygons}
\end{figure}
\end{remark}

\subsection{Canonical polygons}\label{sec:canonical}

Rademacher~\cite{Rademacher1929} asked for a construction of a set of independent generators for congruence subgroups $\G_{0}(N)$. This was answered in~\cite{Rademacher1929,Frasch-1932,Chuman-1973}, but those constructions were not related to fundamental polygons. 
Kulkarni~\cite{Kulkarni-1991} answered this question for all finite index subgroups in the modular group by construction of a special fundamental polygon whose side-pairing transformations form the set of independent generators.

Kulkarni's construction of special polygons for subgroups of the modular group mentioned above~\cite{Kulkarni-1991} is based on Farey symbols and is different from ours. In particular, if the number of cusps is greater than 1, the translation in the half-plane is included as one of the parabolic generators while we specifically exclude it, see~\cite[Figure 2.6]{Kulkarni-1991}, and the sides identified by hyperbolic generators are not necessarily consequent. Kulkarni's polygons were later used by Huang~\cite{Huang1999,Huang2009}, in the context of the realizability of subgroups of prescribed signature, and other authors. Kulkarni's special polygons are compiled from triangles in the Dedekind tessellation by fundamental domains for the action of the extended modular group generated by the modular group and the transformation $z\mapsto \bar z$,  while our construction does not have this restriction.

\begin{defn}\label{def:canonical} Given a signature $(g; m_1,m_2,\dots, m_r;t\ge 1)$, we call a convex polygon~$\Fc$ \emph{canonical} if the following properties hold:
\begin{enumerate}
    \item The number of sides of the polygon is $N:=4g+2r+2(t-1)$, all of infinite length. We choose the order of the sides in the counter-clockwise direction to correspond to the signature, starting from the vertex denoted by \mbox{$V_0=1$}, corresponding to one cusp which is always present. Vertices are labeled counter-clockwise, $V_0, V_1, \dots ,V_N=V_0$.
    \item $4g$ sides come in $g$ quadruples of complete adjacent geodesics. In each quadruple $k$ ($1 \le k \le g$), they are glued in the {commutator manner $b_k^{-1}a_k^{-1}b_ka_k$} by hyperbolic transformations from $\psu$. More precisely, the first side is glued to the third side by $a_k$, and the fourth side is glued to the second side by $b_k$.
    \item $2r$ sides are geodesic rays and come in $r$ pairs. In pair $i$ ($1 \le i \le r$) the rays meet with angle~$\frac{2\pi}{m_i}$ at a vertex and are glued by elliptic transformations. Each elliptic cycle consists of a single vertex.
    \item $2(t-1)$ sides come in $t-1$ pairs of complete adjacent geodesics; they are glued together by parabolic transformations; each parabolic cycle consists of a single vertex (cusp).
    \item\label{def:isom} Each glued side is (an arc of) an isometric circle for its side-pairing transformation.
    \item The following ideal vertices are {equally distributed} on~$\S$: \[ \qquad \arg(V_{4k})=\tfrac{2\pi}{\ell}k \text{ for }0\le k\le g,\,\, \arg(V_{4g+2j})=\tfrac{2\pi}{\ell}(g+j)\text{ for }1\le j\le r+t-1, \] where $\ell:=g+r+t-1$. (\Cref{fig:canonical polygon} has $\ell = 5$, which is the number of sectors.)
\end{enumerate}
\begin{figure}[hbt]
    \includegraphics[width=0.5\textwidth]{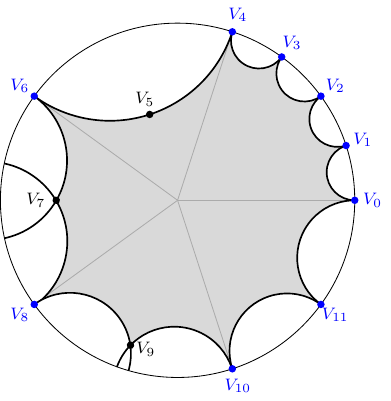}
    \caption{Canonical polygon for signature~$(1;2,3,7;2)$}
    \label{fig:canonical polygon}
\end{figure}
\end{defn}

\begin{remark}\label{rem:uniq} Property \eqref{def:isom} of the above definition implies that the glued sides are arcs of Euclidean circles of the same radius and the point closest to $0 \in \D$ in one circle is mapped to the point closest to $0$ in another \cite[Theorem 3.3.4]{Katok-book}, hence the side-pairing transformations are unique and thus the polygon is \emph{marked}. Therefore, for given signature $(g; m_1,m_2,\dots, m_r;t \ge 1)$,
\begin{itemize}
\item the canonical polygon $\Fc$ exists by a geometric construction (exact formulas for the side-pairing transformations are in Appendix~\ref{sec:explicit}),
\item the canonical polygon $\Fc$ is unique, and
\item the group $\Gamma$ generated by side-pairing transformations of $\Fc$ is unique.
\end{itemize}
\end{remark}

We prove that $\G$ is a Fuchsian group of the first kind with the given signature.

\begin{thm}\label{thm:cqi}Given a signature $(g; m_1,m_2,\dots, m_r;t \ge 1)$ and the marked canonical polygon~$\Fc$ of \Cref{def:canonical}, the  group~$\G$ generated by the side-pairing transformations of~$\Fc$ is a Fuchsian group with this signature and fundamental polygon~$\Fc$. The group~$\G$ is a free product of cyclic groups generated by the independent generators.   
\end{thm}
The proof is based on free combinations of cyclic subgroups and a result of Tukia, which we review before starting the proof.

The canonical polygon~$\Fc$ is related to the free product~\eqref{freeproduct} and is obtained by combining Ford fundamental domains (sets outside of isometric circles) for cyclic subgroups, shown in \Cref{fig:free combinations} and described here: 
\begin{itemize}
    \item  A set outside of two geodesic rays meeting at a point in $\D$ with angle~$\frac{2\pi}{m}$ is a fundamental domain for a finite cyclic group of order $m$ generated by an elliptic M\"obius transformation mapping one geodesic ray to another.
    \item A set outside of two complete geodesics meeting at point in $\S=\partial\D$ is a fundamental domain for an infinite cyclic group generated by a parabolic M\"obius transformation mapping one geodesic to another.
    \item A set outside of two non-intersecting complete geodesics is a fundamental domain for an infinite cyclic group generated by a hyperbolic M\"obius transformation mapping one geodesic to another.
\end{itemize}

\begin{figure}[htb]
    \includegraphics[width=0.93\textwidth]{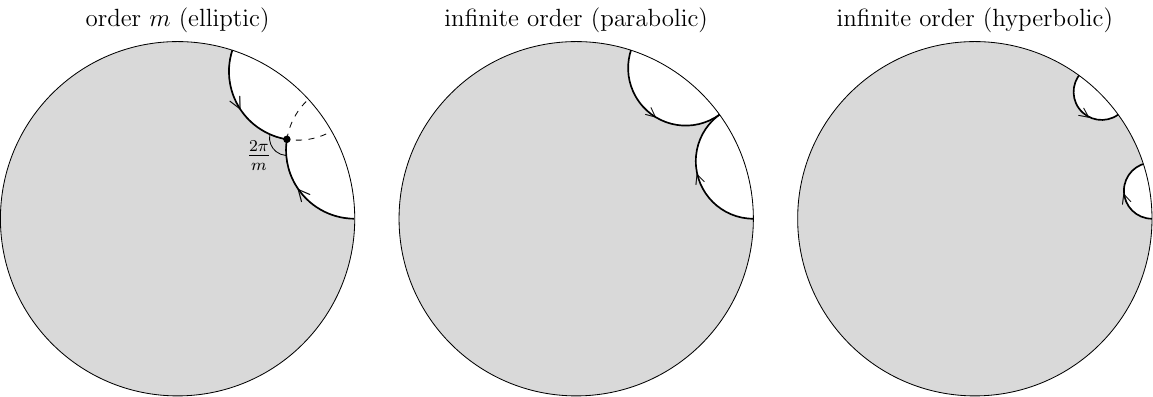}
    \caption{Examples of fundamental polygons used in free combinations}
    \label{fig:free combinations}
\end{figure}


\begin{defn}[Free combination]
Let $\G_j,\,j\in J,$ be a collection of  cyclic  Fuchsian groups, with fundamental polygons $\Fc_j$.
In each case the sides of $\Fc_j$ are identified by generator of $\G_j$ denoted $T_j$.
If \begin{equation} \label{eq:cl sub int} \Cl_\D( \D\setminus \Fc_i  )\subset \Int_\D(\Fc_k)\text{ for }i\ne k, \end{equation} we say that the quadruple $(\G_j, \Fc_j, T_j, J)$ forms a \emph{free combination of cyclic groups}.
\end{defn}

\begin{thm*}[{Tukia~\cite[Theorem~2.6]{Tukia1972}}] Let $(\Gamma_j,\mathcal F_j,T_j,J)$ form a free combination, and $\Gamma$ be the group generated by $\{ \Gamma_j:j \in J \}$. Then the group $\Gamma$ is a free product of groups $\Gamma_j$. Moreover, if either
\begin{enumerate} 
\item the set of fixed points of all $T\in\G$ of infinite order is dense in~$\S$, or
\item the sides of $\mathcal F_j$ are arcs of isometric circles of $T_j$ and $T_j^{-1}$,
\end{enumerate} 
then $\Gamma$ is a discrete subgroup of $\psu$ with fundamental domain $\mathcal F = \bigcap_{j \in J} \mathcal F_j$. \end{thm*}


\begin{proof}[Proof of Theorem~\ref{thm:cqi}]
The sides of the canonical polygon~$\Fc$ can be identified as the sides of fundamental domains of cyclic subgroups of $\G$ as follows:
\begin{itemize}
    \item $g$ quadruples of adjacent complete geodesics define fundamental domains $\Fc_k$ ($1\le k\le g$) for a free group generated by hyperbolic transformations $a_k$ and $b_k$ related to the genus. See also the top right of \Cref{fig:our free combinations} in the appendix.
    \item $r$ pairs of intersecting geodesic rays define
    fundamental domains  for finite cyclic subgroups each generated by an elliptic transformation $\gamma_{j}$ ($4g+1\le j\le 4g+r$);
    \item $t-1$ pairs of complete geodesics meeting at a point of~$\S$ define fundamental domains  for
    infinite cyclic parabolic subgroups generated by parabolic transformations $\g_j$ ($4g+r+1\le j\le 4g+r+t-1$) each fixing a non-congruent cusp different from~$V_0$.
\end{itemize}
The total number of ideal vertices congruent to $V_0$ is $4g+r+t-1$,  ordered counter-clockwise, 
\begin{equation}\label{idealverticesc1}
V_0, V_1,\dots,V_{4g}, V_{4g+2},\dots, V_{4g+2r}, V_{4g+2r+2},\dots, V_{4g+2r+2(t-1)}=V_0,
\end{equation}
and the product
\begin{equation}\label{product}
\prod_{j=4g+1}^{{4g+r+t-1}} \!\! \gamma_{j} \; \prod_{k=1}^g b^{-1}_k a^{-1}_k b_k a_k
\end{equation}
fixes $V_0$.

Relation~\eqref{eq:cl sub int} for the cyclic subgroups 
$\ang{a_k}, \ang{b_k}, \ang{\gamma_j}$  defined above
follows from our construction. Therefore, we have a free combination \begin{equation}\label{quadruple} (\Gamma_j,\mathcal F_j,T_j,J), \end{equation} where $J = \{1,2,...,2g+r+t-1\}$ and each $\G_j$ is either $\ang{a_j}$ or $\ang{b_j}$ or $\ang{\gamma_j}$. According to Tukia's Theorem, the group $\G$ is a free product of cyclic subgroups~\eqref{freeproduct}.  Since the sides of the polygon~$\Fc$ are isometric circles of the independent generators $\ang{a_j}, \ang{b_j}, \ang{\gamma_j}$, by condition (2) of Tukia's Theorem, $\G$ is a  discrete subgroup of $\psu$, and,
\begin{equation*}
    \Fc=\bigcap_{k=1}^\ell \Fc_k,
\end{equation*}
where $\Fc_k$ for $1 \le k \le g$ are intersections of the fundamental polygons for~$\ang{a_k}$ and~$\ang{b_k}$ (see also~\Cref{fig:our free combinations} in the appendix for our $\Fc_k$). Equivalently,~$\Fc$ can be represented as union of ``building blocks,''
\begin{equation*}
    \Fc = \bigcup_{k=1}^\ell \mathcal B(k), 
\end{equation*}
where
\begin{equation*} \label{eq:bb}
\begin{aligned} 
\mathcal B(k) = \{ z \in \mathcal F_{k} :&\, \arg(V_{4(k-1)}) \le \arg(z) \le \arg(V_{4k}) \text{ for } 1\le k\le g\text{, and }\\
&\arg(V_{4g+2(j-1)}) \le \arg(z) \le \arg(V_{4g+2j}) \text{ for } 1\le j\le r+t-1\}\hspace*{-0.5cm}
\end{aligned}
\end{equation*}
are determined using ideal vertices of~\eqref{idealverticesc1}.\footnote{Inequalities between arguments refer to counter-clockwise cyclic order on the circle throughout the paper.}

The hyperbolic area of~$\Fc$ is finite and is equal to 
\[
2\pi\left(2g-2+t+\sum_{i=1}^r \left(1-\tfrac1{m_i}\right)\right)>0
\]
by~\eqref{eq:area}. Therefore, $\G$ is a finitely generated Fuchsian group of the first kind by \cite[Theorem 4.5.2]{Katok-book}.
The polygon~$\Fc$ has $r$ elliptic cycles consisting of a single elliptic vertex with corresponding angles $\frac{2\pi}{m_i}, 1\le i\le r$, and $t-1$ parabolic cycles consisting of a single cusp different from~$V_0$, and one parabolic cycle of $4g+r+t-1$ ideal vertices congruent to~$V_0$. Each cusp, including~$V_0$, is fixed by a parabolic element of~$\G$~\cite[Theorem~4.2.5]{Katok-book}. By discreteness of $\G$, hyperbolic and parabolic elements in $\G$ cannot have a common fixed point~\cite[proof of Theorem 2.4.3]{Katok-book}, hence the product~\eqref{product}, which fixes~$V_0$, is parabolic.

The quotient space $\G\backslash\D$ is decomposed into $r+t$ vertices, $r+2g+(t-1)$ edges, and~$1$ simply connected face. By Euler's formula, its genus satisfies
\[ 2-2g = (r+t) - (r+2g+(t-1))+1 \]
as required. By Poincar\'e's Polygon Theorem,~$\G$~is a Fuchsian group of the first kind with the given signature.
\end{proof}

The explicit construction of the canonical polygon~$\Fc$ for a signature is contained in~Appendix~\ref{sec:explicit}. Although exact formulae were not used in the proofs of Theorems~\ref{thm:bijectivity} and~\ref{thm:attractor}, they were used in computer simulations, and we include them in case they will be useful to others.

\section{Teichm\"uller space}\label{sec:Teich}

\subsection{Fenchel--Nielsen maps}\label{sec:F-N}
There are many equivalent definitions of the Teichm\"uller space of a Riemann orbifold. The following interpretation as a representation space of a Fuchsian group is the most appropriate for our purposes:
\begin{defn}[{\cite[Section 5.4]{Kapovich-book}}] \label{Teich-representation}
Let~$\G$ be a finitely generated Fuchsian group of the first kind. The \emph{Teichm\"uller space}~$\mathcal T(\G)$ is the space of injective  type-preserving representations $\G\to \psu$ with discrete image up to conjugation by an element of $\psu$.
\end{defn}
If two groups $\Gamma_1,\Gamma_2$ have the same signature, there is  a \emph{type-preserving} group isomorphism~$\varphi:\G_1\to\G_2$ and an orientation-preserving  homeomorphism $h:\Gamma_1\backslash\D\to\Gamma_2\backslash\D$ that carries 
bijectively cusps of $\Gamma_1\backslash\D$ to cusps of $\Gamma_2\backslash\D$ fixed by corresponding parabolic elements. 
It follows from the theory of branched covering spaces that
$h$ can be lifted to an orientation-preserving homeomorphism $h:\D\to\D$ such that 
$\Gamma_2=h\circ\Gamma_1\circ h^{-1}$. This is equivalent to the fact that 
the isomorphism~$\varphi$ is \emph{geometric}, i.e., is induced by a homeomorphism $h$,
\begin{equation}\label{eq:geometric}
h(T(z))=\varphi(T)(h(z)) \text{ for any }z\in\D\text{ and }T\in\Gamma_1.
\end{equation} 

Fenchel and Nielsen proved that~$h$ can be extended to the boundary of~$\D$.
The following result is an extension of their theorem; we state it and remark on the historical development afterwards.
\begin{thm}[Fenchel--Nielsen] \label{thm:Deform}
Let $\G_1\in \mathcal T(\G)$. The map $h$ can be extended to the boundary $\S$ and chosen quasi-conformal and isotopic to the identity map.
That is, there is a family of maps $h_t:\bar\D\to\bar\D,\,0\le t\le 1$, with $h_0=\id, h_1=h$ and a family of Fuchsian groups $\G_t$ such that $\G_t=h_t\circ\G\circ h_t^{-1}\in\mathcal T(\G)$.
\end{thm}
We call the map from \Cref{thm:Deform} a \emph{Fenchel--Nielsen map}.

\begin{remark}\label{F-N-remarks}~\begin{enumerate}[label=(\roman*)]
\item Fenchel--Nielsen's Theorem for co-compact~$\G$ (signature $(g;-)$) is contained in their famous manuscript~\cite{Fenchel-Nielsen2003} which, although unpublished until 2003, has had great
influence on the theory of Fuchsian groups. Their approach to the
theorem is through their intersecting axes theorem. Macbeath~\cite{Macbeath1967} generalized it for co-compact~$\G$ with elliptic elements, and Bers~\cite{Bers1972} covered the case of arbitrary signature using quasi-conformal maps. Tukia~\cite[Theorem~3.5 and Corollary~3.5.1]{Tukia1972} rediscovered the significance of the intersecting
axes property and gave a proof of this result in full generality. See also~\cite{Maskit-1976,Maskit-1977,Marden1976} and the bibliographies therein.
\item The Fenchel--Nielsen map $h$ in \Cref{thm:Deform} can be chosen quasi-conformal \cite{Bers1972}. The isotopy to identity is due to contractibility of $\mathcal T(\G)$.
A purely topological proof of this theorem for orbifolds was given by Marden \cite{Marden-1969}. For proofs using analytic methods in which
isotopy is constructed by continuously varying complex dilatation of $h$, see~\cite{Ahlfors-Bers-1960, Tukia-Acta-1985, Douady-Earle-1986, Earle-McMullen-1988}).
    \item The map $h|_\S$ is an orientation-preserving homeomorphism of~$\S$ and therefore preserves the order of the points in the circle.
    \item $h|_\S$ is always $\alpha$-H\"older; and $\alpha=1$ if and only if $h|_\S$ is a M\"obius transformation. If~$h|_\S$ is Lipschitz then it is absolutely continuous. See~\cite{BowenSeries79, Sorvali72}. 
\end{enumerate}
\end{remark}

An alternative definition of the Teichm\"uller space $\mathcal{T}(\G)$ \cite{Bers1972,Earle-McMullen-1988} is shown below; it will be useful at the end of this section.

\begin{defn} \label{defn teich} $\mathcal{T}(\G)$ is the space of pairs $(\G',h)$ where $h:\bar \D\to\bar \D$ is a quasi-conformal Fenchel--Nielsen map such that $\G' = h\circ\G\circ h^{-1}$ modulo the equivalence relation $(\G'_1,h_1)\sim (\G'_2,h_2)$ if there is a conformal map $\mu:\bar \D\to\bar \D$ such that $h_2^{-1}\circ \mu \circ h_1$ is the identity on~$\S$.
\end{defn}

\subsection{\Tdef{} of the canonical polygon}\label{sec:Tdef}

In this section and \Cref{sec:main proof} we use $\Fc^*$ and $\Gamma^*$ for the canonical case and use non-starred versions for non-canonical.

\begin{defn} 
Given a canonical fundamental polygon~$\Fc^*$ and a Fenchel--Nielsen map $h:\bar\D \to \bar\D$, we call the \emph{\Tdef{}} of~$\Fc^*$ the new polygon~$\Fc$ formed by connecting the images of the vertices of $\Fc^*$ under $h$ with geodesic arcs.
\end{defn}
The following theorem shows that this construction is possible. Note that~$\Fc$ is \emph{not} the image of~$\Fc^*$ under~$h$.

\begin{thm} \label{thm:FG-qip} Any Fuchsian group of the first kind with at least one cusp has a fundamental polygon, obtained by \Tdef{} of the canonical polygon, such that the side-pairing transformations of the polygon generate the group.
\end{thm} 

\begin{proof} 
Let~$\G$ be a Fuchsian group of the first kind with signature $(g;m_1,\dots,m_r;t\ge 1)$. As before, let~$\Fc^*$ be the canonical polygon with this signature, $\G^*$ the canonical Fuchsian group generated by side-pairing transformations of~$\Fc^*$, and $h:\bar\D\to\bar\D$ such that $\G=h\circ\G^*\circ h^{-1}$.

Let $\mathcal V^* := \{V^*_1,V^*_2,\dots, V^*_N\}$ be the set of all vertices of~$\Fc^*$ ordered by argument mod~$2\pi$.
Denoting $V_k := h(V^*_k)$ and $\mathcal V := \{V_1,V_2,\dots, V_N\}$, we will show that the points in $\mathcal V$ are also ordered by argument.

If $V^*_k$ and $V^*_{k+1}$ are two consecutive ideal vertices of~$\Fc^*$, then $\arg V^*_k<\arg V^*_{k+1}$, and since $h|_\S:\S\to\S$ is an orientation-preserving homeomorphism, $\arg V_{k}<\arg V_{k+1}$. 
Now let $V^*_k$ be an elliptic vertex of order $m$ such that $\arg V^*_{k-1}<\arg V^*_k<\arg V^*_{k+1}$ (notice that $V^*_{k-1}$ and $V^*_{k+1}$ are ideal vertices). Connect $V_{k-1}$ with $V_k$ and $V_k$ with $V_{k+1}$ by geodesic rays. 
By the previous argument, $\arg V_{k-1}<\arg V_{k+1}$. We now show that $\arg(V_{k-1}) < \arg(V_k) < \arg(V_{k+1})$.

Let $\gamma^* \in\G^*$ be the elliptic element of order $m$ that fixes $V^*_k$ and maps geodesic ray $(V^*_{k-1},V^*_k)$ to geodesic ray $(V^*_{k+1},V^*_k)$. The angle between these geodesic rays in the clockwise direction is equal to $\frac{2\pi}m$. By~\eqref{eq:geometric}, $\g := h\circ \gamma^*\circ h^{-1}\in \Gamma$ is also an elliptic element of order $m$, and $\g(V_{k-1})=V_{k+1}$,  $\g(V_k)=V_k$. Thus, $\g$ maps the geodesic ray $(V_{k-1},V_k)$ to the geodesic ray $(V_{k+1}, V_k)$. Since~$h$ is an orientation-preserving homeomorphism of $\bar\D$ isotopic to identity (\Cref{thm:Deform}), the angle between these geodesics in the clockwise direction cannot jump to a multiple of $\frac{2\pi}{m}$ and must be also $\frac{2\pi}{m}$.

Assume, by contradiction, that $\arg(V_k)\ge \arg(V_{k+1})$, as shown in \Cref{fig:contradiction}. Connect $V_{k-1}$ with $V_k$ and $V_k$ with $V_{k+1}$ with geodesics.
Since $\arg(V_{k+1}) > \arg(V_{k-1})$ and since the geodesic through $V_{k+1}$ and $V_k$ (which is an Euclidean circle orthogonal to~$\S$) intersects the geodesic through  $V_{k-1}$ and $V_k$ (which is also an Euclidean circle orthogonal to~$\S$) in $V_k$, the angle at $V_k$ between geodesic rays $V_{k-1}V_k$ and $V_{k+1}V_k$ in clockwise direction is greater than $\pi$, as shown in the right of \Cref{fig:contradiction}, a contradiction.
This proves $\arg(V_k) < \arg(V_{k+1})$. By similar reasoning, $\arg(V_{k-1}) < \arg(V_k)$ as well.
Therefore the vertices $V_1,V_2,...,V_N$ are also ordered by argument.
Connecting them in order by geodesics, we obtain a polygon~$\Fc$ with the same combinatorial structure as~$\Fc^*$ whose side-pairing transformations generate~$\G$ by~\eqref{eq:geometric}.
\end{proof}

\begin{figure}[htb]
\includegraphics[width=0.9\textwidth]{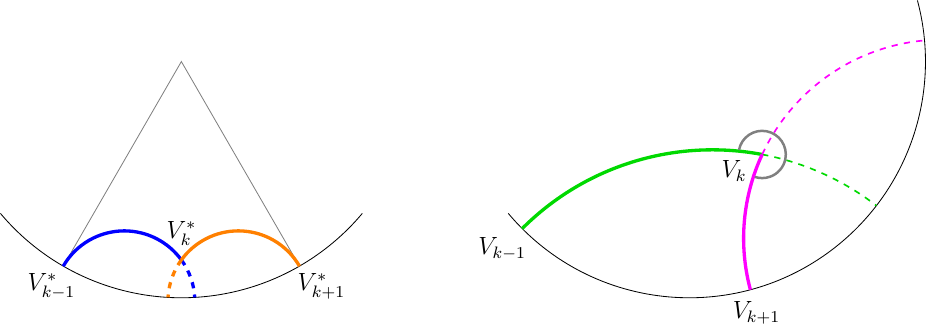}
\caption{Argument by contradiction for order of vertices}
\label{fig:contradiction}
\end{figure}

\begin{defn} \label{def:qi} The marked fundamental polygon constructed by \Tdef{} (\Cref{thm:FG-qip}) is called \emph{quasi-ideal}. The marking is given by the system of generators. It is characterized by the following properties:
\begin{enumerate}
    \item the number of sides of the polygon is $N:=4g+2r+2(t-1)$, all of infinite length;
    \item $4g$ sides come in $g$ quadruples of complete adjacent geodesics glued in the commutator manner by hyperbolic transformations from $\psu$;
    \item $2r$ sides come in $r$ pairs of geodesic rays glued by elliptic transformations (for $1\le i\le r$, the pair of rays meet with angle $\frac{2\pi}{m_i}$ at a vertex); 
    \item $2(t-1)$ sides come in $t-1$ pairs of complete adjacent geodesics glued by parabolic transformations; 
    \item[(5*)] the product of commutators of hyperbolic transformations and elliptic and parabolic transformations fixing $V_0$ is parabolic.
\end{enumerate}
Properties (1)--(4) hold by construction, and property (5*) holds since $h$ is a type-preserving homeomorphism.
\end{defn}

The following theorem may be considered a converse to \Cref{thm:FG-qip}:

\begin{thm}\label{thm:qip-FG} Given a signature $(g; m_1,m_2,\dots, m_r;t\ge 1)$ and a marked quasi-ideal polygon~$\Fc$, the group~$\G$ generated by the side-pairing transformations of~$\Fc$ is a Fuchsian group of the first kind with this signature, and~$\G$ is a free product of cyclic groups generated by the independent generators.
\end{thm}
\begin{proof} 
The quasi-ideal polygon $\Fc$ is a \Tdef{} of a canonical quasi-ideal polygon $\Fc^*$. We now denote the quadruple~\eqref{quadruple} for the canonical polygon by $(\G^*_j, \Fc^*_j, T^*_j, J)$. By construction, $\G_j = h \circ \G_j^* \circ h^{-1}$ are cyclic groups obtained from independent generators $T_j = h \circ T_j^* \circ h^{-1}$, with fundamental domains $\Fc_j$ that are \Tdef{} of $\Fc_j^*$. Therefore, the quadruple $(\G_j, \Fc_j, T_j, J)$ forms a free combination of cyclic groups.

The cycle relations, including the classical parabolic cusp condition (5*), are satisfied, so by Poincar\'e's Polygon Theorem~$\G$ is a Fuchsian group. The finiteness of the hyperbolic area of~$\Fc$ implies that it is of the first kind~\cite[Theorem~4.5.2]{Katok-book}, in which case the limit set of~$\G$ is the entire~$\S$. Since the limit set of~$\G$ is the closure of fixed points of the hyperbolic elements of~$\G$ \cite[Theorem~3.4.4]{Katok-book}, the closure of fixed points of the hyperbolic elements is already dense in~$\S$, hence condition (1) of Tukia's Theorem is satisfied and the result follows.
\end{proof}

The following theorem is not used for the proof of the main theorem, but it is of independent interest.
\begin{thm} \label{thm:teich-homeo}
Given a signature $(g;m_1,\dots,m_r;t\ge 1)$, the space of marked quasi-ideal polygons for that signature (up to orientation-preserving isometries of~$\D$) is homeomorphic to the Teichm\"uller space $\mathcal{T}(\G)$ of Fuchsian groups~$\G$ with this signature, and therefore to $\R^{6g-6+2(r+t)}$.
\end{thm}

\begin{proof} The space of all marked quasi-ideal polygons with a given signature  is endowed with the following topology: for quasi-ideal polygons $\{P_n\}$ and $P$, $P_n\to P$ if all vertices and side-pairing transformations of $P_n$ converge to the vertices and side-pairing transformations of $P$. We denote this space modulo $\psu$ by $\mathcal{P}(\G)$. The space $\mathcal T(\G)$  is endowed with topology coming from the injective representation $\G\to \psu$. We will now show that the maps of Theorems~\ref{thm:FG-qip} and~\ref{thm:qip-FG} are bijective.

Consider the map $\mathcal{T}(\G)\to \mathcal{P}(\G)$ of \Cref{thm:FG-qip}.
Suppose \Tdef{}s of two Fuchsian groups $\G_1$ and $\G_2$ produce two polygons $\Fc_1$ and $\Fc_2$ which represent the same point in $\mathcal{P}(\G)$, that is, there is a M\"obius transformation $\mu\in\psu$ such that $\mu\circ h_1$ and $h_2$ agree on all vertices of the canonical quasi{}-ideal polygon~$\Fc^*$. This means that $h=h_2^{-1}\circ\mu\circ h_1$ fixes all the vertices of~$\Fc^*$ that include all its cusps, and $h\circ T_i\circ h^{-1}=T_i$ for all side-pairing transformations of $\G^*$. Therefore, $h\circ T\circ h^{-1}=T$ for all $T\in\G^*$, that is, $h$ commutes with all elements of $\G^*$. Each parabolic fixed point of $\G^*$ is $T(V)$ for some $T\in\G^*$ and a cusp $V$ of~$\Fc^*$, and we conclude that $h(T(V))=T(h(V))=T(V)$, i.e., $h=h_2^{-1}\circ\mu\circ h_1$ fixes all parabolic fixed points of $\G^*$. Since parabolic fixed points  are dense in~$\S$ \cite[Theorem 5.3.9]{Beardon-book} and by the continuity of the Fenchel--Nielsen maps, $h_2^{-1}\circ\mu\circ h_1$ is the identity on~$\S$. By \Cref{defn teich}, this implies that $\G_1$ and $\G_2$ represent the same point in $\mathcal{T}(\G)$, which shows injectivity. Surjectivity follows from \Cref{def:qi}.

Since each quasi-ideal polygon (\Cref{def:qi}) satisfies properties (1)--(4) and (5*), the map of \Cref{thm:qip-FG} may be considered as the map $\mathcal{P}(\G)\to \mathcal{T}(\G)$.
Suppose two quasi-ideal polygons $\Fc_1$ and $\Fc_2$ are obtained by \Tdef{} of~$\Fc^*$. Then the groups of side-pairing transformations of $\Fc_1$ and $\Fc_2$ are $\G_1=h_1\circ \G^*\circ h_1^{-1}$ and $\G_2=h_2\circ \G^*\circ h_2^{-1}$, respectively. If they represent the same point in $\mathcal{T}(\G)$, by \Cref{Teich-representation}, there exists a M\"obius transformation~$\mu$ which conjugates them, i.e., $\mu\circ \G_1\circ \mu^{-1}=\G_2$, and therefore
\[
(h_2^{-1}\circ\mu\circ h_1)\circ \G^*\circ (h_2^{-1}\circ\mu\circ h_1)^{-1}=\G^*.
\]
By \Cref{thm:FG-qip} if $V$ is the vertex of~$\Fc^*$, then 
$(h_2^{-1}\circ\mu\circ h_1)(V)=\tilde{V}$ is also a vertex of~$\Fc^*$,  and we have
\[
\mu\circ h_1(V)=h_2(\tilde{V}),
\]
that is, $\mu$ maps vertices of $\Fc_1$ to vertices of $\Fc_2$. This implies that $\Fc_1$ and $\Fc_2$ represent the same point in $\mathcal{P}(\G)$, which shows injectivity. Surjectivity follows from \Cref{thm:FG-qip} and \Cref{def:qi}.

This bijection is clearly continuous.
Hence, $\mathcal{P}(\G)$ is homeomorphic to $\mathcal T(\G)$ and we have an alternative definition of the Teichm\"uller space as the space of marked quasi-ideal polygons with a given signature modulo $\psu$. 

The dimension can be easily calculated using quasi-ideal polygons: $4g+r+t-1$ ideal vertices congruent to $V_0$ are given by $4g+r+t-1$ real parameters, hyperbolic generators are given by $2g$ real parameters representing marking, elliptic vertices are given by $r$ real parameters, and additional cusps by $t-1$ real parameters. Notice that elliptic and parabolic generators are uniquely determined by the positions of their fixed points. Since the product fixing $V_0$ must be parabolic, we need to decrease the number of parameters by $1$. Thus the space of marked quasi-ideal polygons modulo orientation-preserving isometries of~$\D$  (which is $3$-dimensional) is given by
\[
(4g+r+t-1)+(2g)+r +(t-1)-1-3=6g-6+2(r+t)
\]
real parameters which depend continuously on  the  marked quasi-ideal polygon. This can be used to see the homeomorphism $\mathcal{P}(\G) \cong \R^{6g-6+2(r+t)}$ directly. \end{proof}

\section{Boundary maps and their natural extensions}\label{sec:boundary map}

Let $\Gamma$ a Fuchsian group of signature $(g; m_1,m_2,\dots, m_r;t\ge 1)$ and let~$\Fc$ be its associated quasi-ideal $N$-gon obtained by \Tdef{} of a canonical polygon (see \Cref{sec:Tdef}), with $N=4g + 2r + 2(t-1)$.

Let $\mathcal V := \{V_1,V_2,\dots, V_N\}$ be the set of all vertices of~$\Fc$ ordered by argument mod~$2\pi$ (the indices are mod $N$, so $V_N = V_0$).
The polygon~$\Fc$ also has $N$ sides, each connecting two consecutive vertices: side~$k$ is $V_{k-1}V_k$.

We now describe a family of boundary maps associated with $\Gamma$ and~$\Fc$.  Let $A_1,...,A_N \in \S$ be points satisfying
\begin{equation*} \label{eq:partition} A_k \in (V_{k-1},V_{k+1}) \text{ if }V_k \text{ is elliptic and }A_k = V_k\text{ otherwise}.
\end{equation*}
We call the former \emph{elliptic partition points} and the latter \emph{ideal partition points}.

\medskip
For each partition $\A := \{A_1,...,A_N\}$, denoting $A_0 = A_N$ as with the vertices, we define the \emph{boundary map} $f_{\A}:\S \to \S$ by
\begin{equation} \label{eq:boundary map}
f_{\A}(x) = \gamma_k(x) \quad\text{for }x \in [A_{k-1}, A_{k}) \subset \S, \end{equation}
where $\gamma_k:\bar\D \to \bar\D$ is the transformation that maps side $V_{k-1}V_{k}$ to its paired side (see \Cref{def:qi} for the description of how sides are glued).

The \emph{natural extension} of the boundary map $f_{\A}:\S\to\S$ is the map~$F_{\A}$ on the space $\S \times \S \setminus \Delta = \{(u,w)\in\S\times\S : u \ne w\}$ given by
\begin{equation}\label{eq:natural extension}
  F_{\A}(u,w) = \big( \gamma_k(u),\gamma_k(w) \big) \quad\text{for }w \in [A_{k-1}, A_{k}),
\end{equation}
using the same partition~$\A$ as the boundary map, and the global attractor 
\[
    \overline{\bigcap_{n=0}^\infty F_{\A}^n(\S \times \S \setminus \Delta)}
\] 
is a bijectivity domain to which every point is mapped.
Technically, the \emph{natural extension} in the classical sense is the restriction of $F_{\A}$ to this set.
\begin{remark}\label{reduction map} If one identifies a pair $(u,w)\in \S \times \S \setminus \Delta$ with a geodesic in~$\D$ from~$u$ to~$w$, then $F_{\A}$ maps geodesics to geodesics. The attractor of $F_{\A}$ corresponds to the set of \emph{reduced geodesics}, therefore $F_{\A}$ is also called the \emph{reduction map}. Notice that, since $F_\A$ acts by generators of $\G$, all geodesics on $\D$ obtained by $F_\A$ from a given one represent the same geodesic on $\G\backslash \D$.
This approach has applications to symbolic coding of geodesics on $\G\backslash\D$ (see also Section \ref{sec:remarks}).
\end{remark}

\subsection*{Cycle property} \label{sec:cycle}
In the study of one-dimensional piecewise continuous maps, understanding the orbits of the discontinuity points is often crucial. Although formula~\eqref{eq:boundary map} defines $f_{\A}(A_k)=\gamma_{k+1}(A_k)$, one can consider two orbits of $A_k$, the ``upper'' orbit starting with $\gamma_{k+1}(A_k)$, and the ``lower'' orbit starting with $\gamma_{k}(A_k)$.
We first consider the ``upper'' and ``lower'' orbits of ideal partition points.
\begin{prop} \label{orbit-ideal} The upper- and lower-orbits of any ideal vertex/partition point of~$\Fc$ is finite. 
\end{prop}
\begin{proof} Ideal vertices are cusps of~$\Fc$. By construction of~$\Fc$, each parabolic cycle non-congruent to $V_0$ consists of a single cusp which is fixed by the corresponding parabolic element of~$\G$. Therefore its orbit consists of a single point, i.e., is finite. The rest of the cusps are congruent to $V_0$, so it is sufficient to consider only $V_0$ itself. The upper and lower orbits of $V_0$ consist of the $N - r - (t-1) = 4 g + r + t - 1$ cusps congruent to it.
\end{proof}
The following patterns play an essential role in studying the boundary map $f_{\A}$ (see \cite{KU-ab-structure,KU-structure}).\footnote{For a general piecewise continuous map~$f$, one can define functions $f^+$ and $f^-$ taking different conventions for behavior at discontinuity points anywhere in the orbits. For our system it is impossible to hit another discontinuity point before the end of the cycle, so such details are unnecessary.}
\begin{defn} \label{defn:cycle property} We say that $A_k$ has the \emph{cycle property} (also called ``matching condition'' in some works) if there exist non-negative integers $i_k$, $j_k$ such that 
\begin{equation}\label{eq:cycle}
    f^{i_k}_{\A}(\gamma_{k+1}A_k) 
    = f^{j_k}_{\A}(\gamma_{k}A_k).
\end{equation}
\end{defn}
Let $V_k$ be an elliptic vertex of order $m$. Then $V_{k-1}=A_{k-1}$ and $V_{k+1}=A_{k+1}$ are ideal partition points and 
\begin{equation}\label{eq:deffA}
\begin{array}{ll}
f_{\A}(x)= \gamma_{k}(x) &\text{ if }x\in [A_{k-1},A_k)\\  
f_{\A}(x)=\gamma_{k+1}(x) &\text{ if }x\in [A_k,A_{k+1}), 
       \end{array}
\end{equation}
where $\gamma_{k}=c_m$ is a {\bf clockwise} hyperbolic rotation around the elliptic vertex $V_k$ by the angle $\frac{2\pi}{m}$ for some $m\ge 2$,
and $\gamma_{k+1}=c_m^{-1}$. 
\begin{thm}\label{thm:cycleprop}
Let $V_k$ be an elliptic vertex of order $m$. An elliptic partition point $A_k \in (V_{k-1},V_{k+1})$ has the cycle property if $c_m^j(A_k)\ne V_{k-1}$ for all integers $j$.
\end{thm}

\begin{proof}
Since $c_m^m=\id$, the $c_m$-orbit of each point in the closure of the unit disk $\bar\D$ is finite, so the orbit of $A_k$ under $c_m$ is a finite set:
\[
    A_k, c_m(A_k),\dots,  c_m^{m-1}(A_k).
\]
The equation $c_m^m(A_k)=A_k$ can be rewritten as
\begin{equation}\label{eq:break}
    c_m^j(A_k)=(c_m^{-1})^{m-j}(A_k)
\end{equation}
for any $0\le j\le m$. We will show that there is a unique way to decompose this equation which is compatible with the definition of $f_{\A}$~\eqref{eq:deffA}.
Since the rotation is monotone, there exists a non-negative integer $j_k$ such that 
\[
    f_{\A}^{j_k}(\gamma_{k}A_k)=c_m^{j_k}(\gamma_{k}A_k)=c_m^{j_k+1}(A_{k})
\]
is not in the open interval $(V_{k-1}, V_{k+1})$, and since the angle between geodesics $V_kV_{k-1}$ and $V_kV_{k+1}$ in the clockwise direction is equal to $\frac{2\pi}{m}$, such $j_k$ is unique mod $m$. This means that for all $j_k+1<j\le m$, $c_m^j(A_k)\in [A_k, V_{k+1})$, where $f_{\A}=c_m^{-1}$, and
the equation~\eqref{eq:break} with $j=j_k+1$ gives us exactly the equation~\eqref{eq:cycle} with $i_k=m-j_k-2$.
In this case $f^{j_k}_{\A}(\gamma_{k}A_k)\in (V_{k-1},V_{k+1})$ (clockwise direction), and we have the cycle with the end of the cycle $c_m^{j_k+1}(A_k)$, which is in the complement to the open interval $(V_{k-1}, V_{k+1})$. 
\end{proof}
The situation when $c^{j}_m(A_k)=V_{k-1}$ can be also easily described. The cycle property does not hold, but the orbits are still finite. This only occurs in two situations: when $m$ is odd and $A_k = M_k$, or $m$ is even and $A_k = P_k$ or $Q_k$.

\begin{prop}\label{prop:finiteorbit}
If $c^{j}_m(A_k)=V_{k-1}$ for some $j$, 
then the upper- and lower-orbits of~$A_k$ under $f_{\A}$ are finite.
\end{prop}
\begin{proof} We have
\[
 c_m^{j-1}(\gamma_{k}A_k)=f^{{j-1}}_{\A}(\gamma_{k}A_k)=V_{k-1} \quad\text{and}\quad c_m^{-i}(\gamma_{k+1}A_k)=f^{{i}}_{\A}(\gamma_{k+1}A_k)=V_{k+1} 
\]
with $i=m-j-2$. Since  $V_{k-1}$ and $V_{k+1}$ are ideal vertices/partition points, their orbits under $f_{\A}$ are finite by \Cref{orbit-ideal}, hence $A_k$ has finite upper- and lower-orbits.
\end{proof}

\section{Proof of Main Theorem for canonical case}\label{sec:canonical proof}

We will now build an explicit description of the attractor of the natural extension associated to the \emph{canonical} quasi-ideal polygon with notation $\A$ instead of $\A^*$ in this section. Using additional details in the appendix, we describe a set~$\Omega_{\A}$ (see~\eqref{omega}). We first show that $F_{\A}$ is bijective on this set and then show that it is the global attractor for $F_{\A}:\S \times \S \setminus \Delta \to \S \times \S \setminus \Delta$.

For any elliptic vertex~$V_k \in \D$, we additionally label points~$P_k, Q_k \in \S$ such that~$V_k$ is the intersection of the geodesics~$V_{k-1}Q_k$ and~$P_kV_{k+1}$. See \Cref{fig:VPAQV} for examples of these labels. For ideal~$V_k$, we have $P_k = Q_k = V_k$.

\begin{figure}[hbt]
	\begin{tikzpicture}[scale=8,thick]
		\draw (-45:1) arc (-45:-135:1);
		\draw (-120:1.045) node {$V_{k-1}$};
		\draw (-100:1.04) node {$P_k$};
		\draw [red, very thick] (-96:0.98) -- (-96:1.02) (-96:1.055) node {$A_k$};
		\draw (-80:1.048) node {$Q_k$};
		\draw (-60:1.06) node {$V_{k+1}$};
		\draw (-0.01,-0.69) node {$V_k$};
		\draw (-120:1) arc (150:10:0.36397023426620234);
		\draw (-100:1) arc (170:30:0.36397023426620234);
	\end{tikzpicture}

    \caption{Labels of points on the boundary}
    \label{fig:VPAQV}
\end{figure}
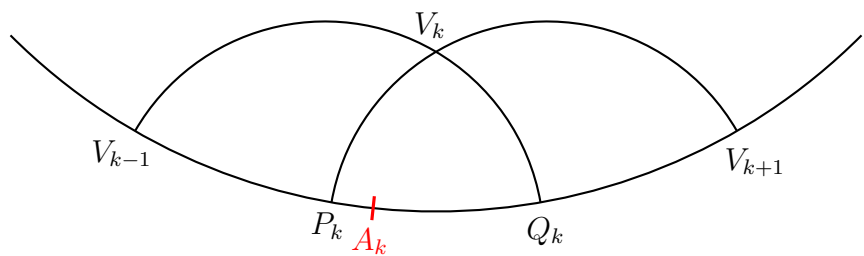

Recall that $V_0=1$. We will make frequent use of the point 
\begin{equation}\label{v} v := e^{(\pi/\ell)i}. \end{equation}
If $V_1$ is parabolic, then $v = V_1$. If $V_1$ is hyperbolic, then $v = V_2$ (see \Cref{fig:canonical polygon}).
If $V_1$ is elliptic, then $v$ is the projection $\frac{V_{1}}{\abs{V_{1}}}$ and we also denote
\begin{equation}\label{pq}
p := P_{1} \qquad \text{and} \qquad
q := Q_{1}.
\end{equation}
Notice that $v$ is the midpoint between~$p$ and~$q$ and $v^2 = V_2$ (see \Cref{fig:flower}).

\begin{figure}[hbt]
  \includegraphics[width=0.95\textwidth]{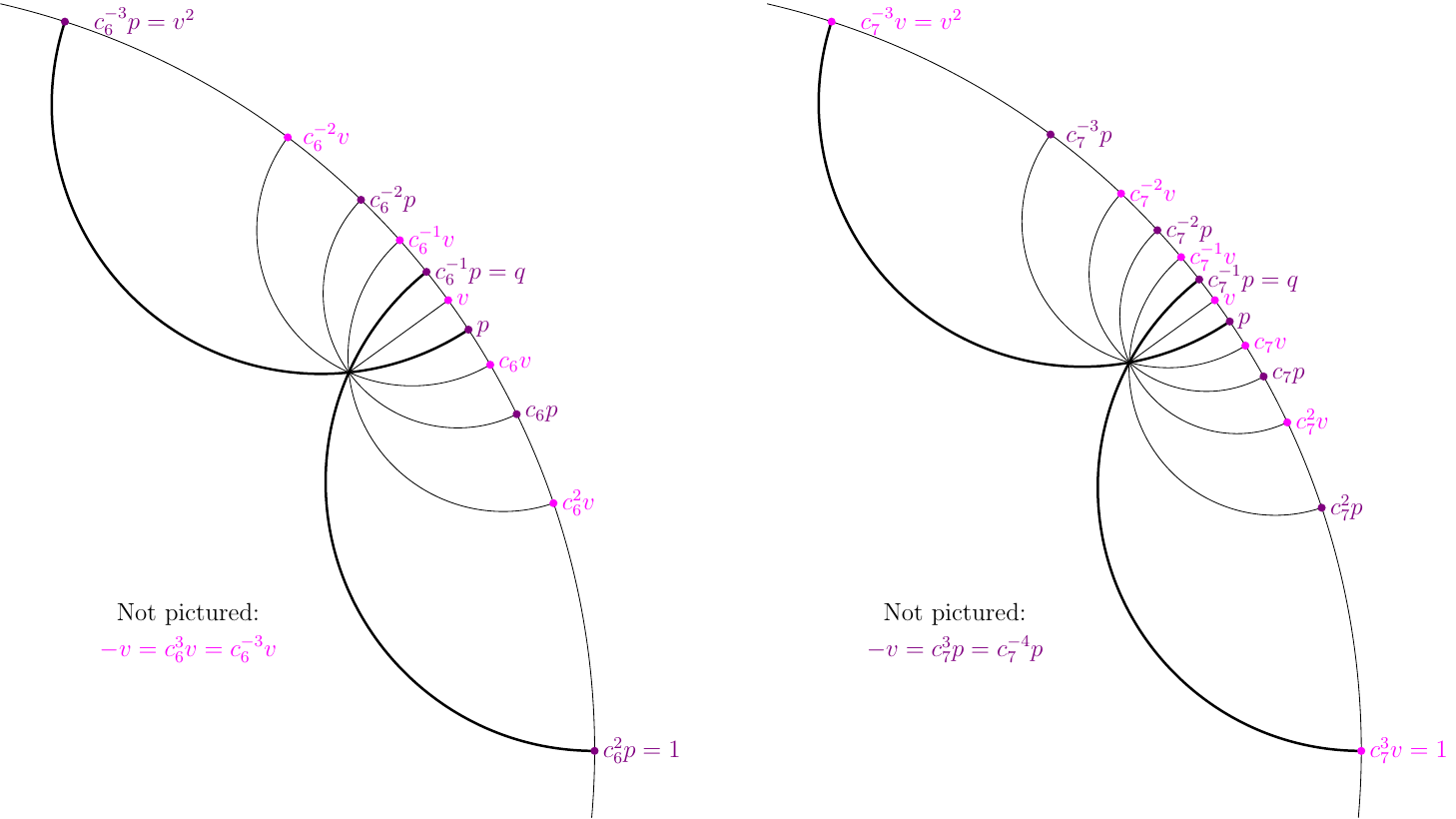}
    \caption{Orbit of $v$ and $p$ under $c_m$ for $m=6$ (left), $m=7$~(right)}
    \label{fig:flower}
\end{figure}

\subsection{Special Markov partitions} \label{sec:markov}

There are three partitions of particular importance:
\begin{itemize}
    \item Left: each $A_k = P_k$.
    \item Right: each $A_k = Q_k$.
    \item Midpoint: each $A_k = M_k$, the midpoint of the circular arc $[P_k,Q_k]$.
\end{itemize}
We denote these partitions as $\mathcal P$, $\mathcal Q$, and $\mathcal M$, respectively. (Recall that if $V_k$ is ideal then $P_k = M_k = Q_k = V_k$.)

\begin{prop}\label{prop:markov}
The maps $f_{\mathcal P}$, $f_{\mathcal Q}$, and $f_{\mathcal M}$ have finite Markov property.
\end{prop}
\begin{proof}
The finite Markov property on each of the three partitions $\mathcal P, \mathcal Q, \mathcal M$ (technically sofic, with added images of the elliptic partition points) follows from the finiteness of upper- and lower-orbits of the corresponding partition points (\Cref{prop:finiteorbit}). 

Note that $M_k$ is the point in $[P_k,Q_k]$ such that $\angle P_kV_kM_k=\angle M_kV_kQ_k=\frac{\pi}{m}$, that is, where the angle bisector of $P_kV_kQ_k$ meets the boundary~$\S$.

If $A_k=M_k$ is an elliptic partition point in $\mathcal B(n)$, then the end of the cycle is the second endpoint of the geodesic through $V_k$ and $M_k$, which (since~$\Fc$ is canonical) is~$-M_k$. 
\begin{itemize}
    \item If $\ell$ is even, $-M_k$ is the middle point of $\mathcal B(n+\frac{\ell}2)$, which is either an elliptic partition point in this building block or an ideal vertex. 
    In the latter case, \Cref{orbit-ideal} applied to $-M_k$ implies that the upper- and lower-orbits of $M_k$ are finite. In the former case, the orbit of $M_k$ is still finite because the union of all midpoints and ideal points is finite.
    \item If $\ell$ is odd, $-M_k=e^{2\pi i(\floor{\ell/2}+1)/\ell}$ is an ideal vertex. Again, by \Cref{orbit-ideal} applied to $-M_k$ the upper- and lower-orbits of $M_k$ are finite.
\end{itemize}
If $A_k$ is either $P_k$ or $Q_k$, and $m$ is even, then by \Cref{prop:finiteorbit} the upper and lower orbits are finite. If $m$ is odd, then the end of the cycle will be $-M_k$, and the orbits are finite by the argument in the previous paragraph.
\end{proof}

\FloatBarrier
\subsection{Domain of bijectivity}
\newcommand\Strip[1]{ \Omega_{\A}^{(#1)} }

In this section we define a bijectivity domain $\Omega_{\A} \subset \S \times \S$ for the map $F_{\A}$ from~\eqref{eq:natural extension} using some additional notation.

Given a signature $(g; m_1,m_2,\dots, m_r;t\ge 1)$, we consider the string
\begin{equation} \label{string}
    s \;\;=\;\;
    [\STRING]
    \;\;=\;\;
    \{\underbrace{\square,\dots,\square}_{\text{\scriptsize$g$}}, m_1,\dots, m_r, \underbrace{\infty,\dots,\infty}_{\text{\scriptsize$t-1$}}\}
\end{equation}
in the alphabet $\{2,3,4,\dots\}\cup\{\infty,\square\}$, with $m_i \le m_{i+1}$, where $\ell=g+r+t-1$. We also denote the rotation $R: \bar\D \to \bar\D$ around the origin by 
\begin{equation} \label{R}
    R(z) = e^{(2\pi/\ell)i} z.
\end{equation}
Note that $R \times R$ is a diagonal translation in $\S \times \S$.

The set
\begin{equation} \label{omega} \Omega_{\A} := \bigcup_{k=1}^{\ell} (R^{k-1} \times R^{k-1}) \Strip{s_k} \end{equation}
is composed of several horizontal strips, where each strip $\Strip{s_k} \subset \S \times [1,e^{(2\pi/\ell)i}]$ is defined below in~\eqref{eq:horizontal strips par-hyp} or~\eqref{eq:horizontal strips}. See~\Cref{fig:attractor example} for an example of $\Omega_{\A}$ with the horizontal strips $(R^{k-1} \times R^{k-1}) \Strip{s_k}$ in different colors on the left; their images on the right of \Cref{fig:attractor example} are vertical strips.

\begin{figure}[htb]
\includegraphics[width=0.9\textwidth]{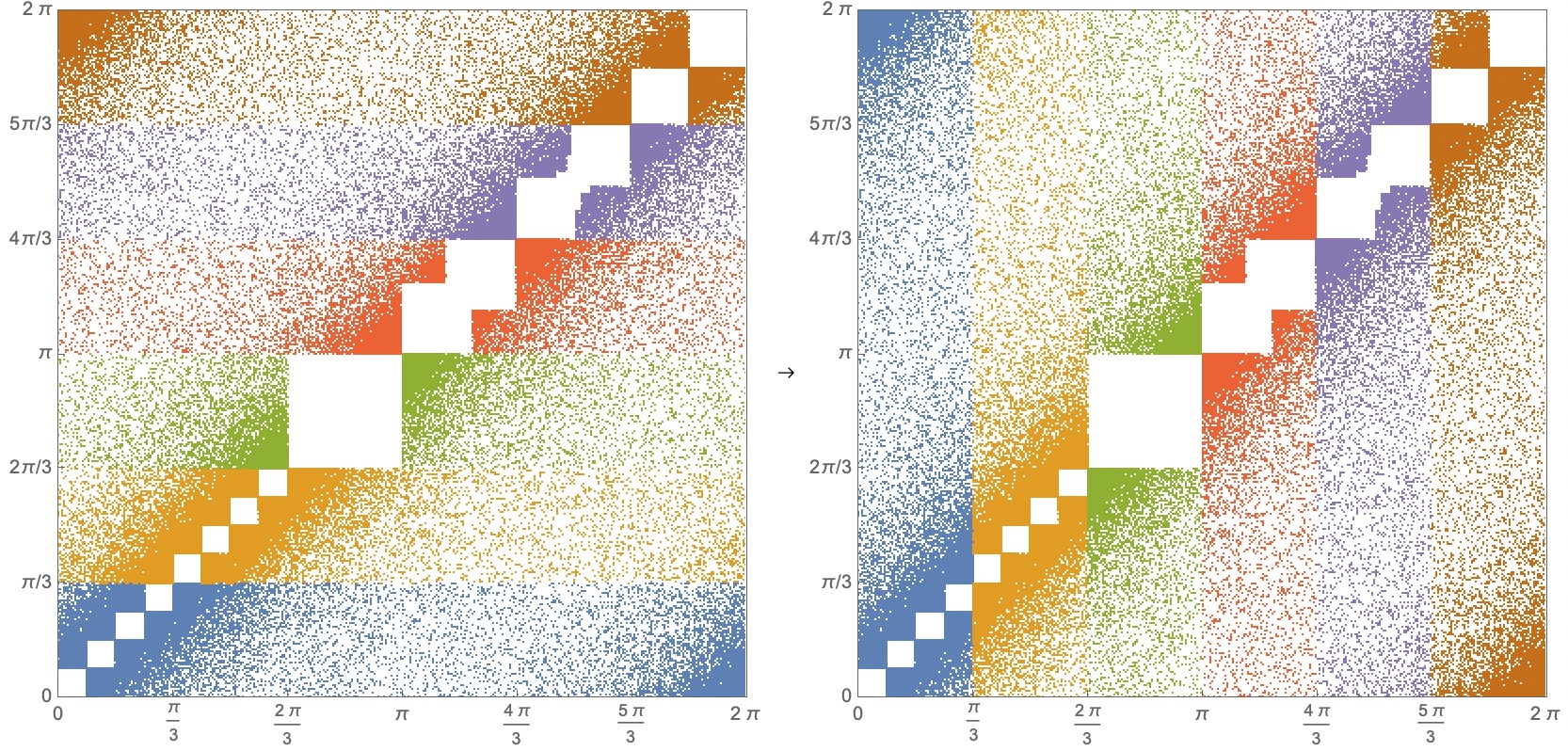}
\caption{Bijectivity domain (attractor)~$\Omega_{\A}$ for $s = \square\square258\infty$, and its image}
\label{fig:attractor example}
\end{figure}

For $s_k \in \{\square,2,\infty\}$, the horizontal strip $\Strip{s_k}$ does not depend on the partition $\A$ and, recalling $v := e^{(\pi/\ell)i}$, is given by
\begin{equation}\label{eq:horizontal strips par-hyp}
\begin{aligned}
    \Strip \square &= ([v^{1/2},1] \times [1,v^{1/2}]) \cup ([v,v^{1/2}] \times [v^{1/2},v]) \\ &\qquad \cup ([v^{3/2},v] \times [v,v^{3/2}]) \cup ([v^2,v^{3/2}] \times [v^{3/2},v^2]) \\
    \Strip 2 &= [v^2, 1] \times [1, v^2] \\
    \Strip \infty &= \big( [v, 1] \times [1, v] \big) \cup \big( [v^2, v] \times [v, v^2] \big),
\end{aligned}
\end{equation}
where the interval notations represent clockwise intervals in the circle~$\S$.

For $s_k = m \ge 3$, the strip $\Strip{s_k}$ is more difficult to describe because it does depend on $\A$, in particular, on the elliptic partition point in the building block~$\mathcal B(k)$. It is sufficient to describe this for $k=1$. Let $a := A_{1}\in (1, v^2)$. 
From now on, we denote
\[ I := {i}_{1} \qquad\text{and}\qquad J := {j}_{1} \]
from \Cref{thm:cycleprop} if $a$ has the cycle property, or $I=i-1$ and $J=j$ in the special cases addressed in \Cref{prop:finiteorbit}. Note that 
\[ \frac{2\pi}{\ell}>\arg(a)>\arg(c_m(a))>\cdots >\arg(c_m^{J}(a)) \ge 0. \]
The strip $\Strip{m}$ consists of several rectangles, as described here:
\begin{equation}\label{eq:horizontal strips}
\begin{aligned}
    \mathcal U_{m}(i) &= [v^2, c_m^{-i+1}(1)] \times [c_m^{-i+1}(a), c_m^{-i}(a)], \quad 1 \le i \le I \\
    \mathcal U_{m}(I+1) &= [v^2, c_m^{-I}(1)] \times [c_m^{-I}(a), v^2] \\
    \mathcal L_{m}(j) &= [c_m^{j-1}(v^2), 1] \times [c_m^j(a), c_m^{j-1}(a)], \quad 1 \le j \le J \\
    \mathcal L_{m}(J+1) &= [c_m^{J}(v^2), 1] \times [1, c_m^{J}(a)] \\
    \Strip{m} &= \bigcup_{i=1}^{I+1}\mathcal U_{m}(i) \cup \bigcup_{j=1}^{J+1} \mathcal L_{m}(j).
\end{aligned}
\end{equation}
\Cref{fig:strip 10} shows $\Strip{10}$ as a union of $\mathcal L_{10}(j)$ and $\mathcal U_{10}(i)$ rectangles, and \Cref{fig:strip 11} shows the same for $m=11$.

\begin{figure}[hb]    \includegraphics[width=0.95\textwidth]{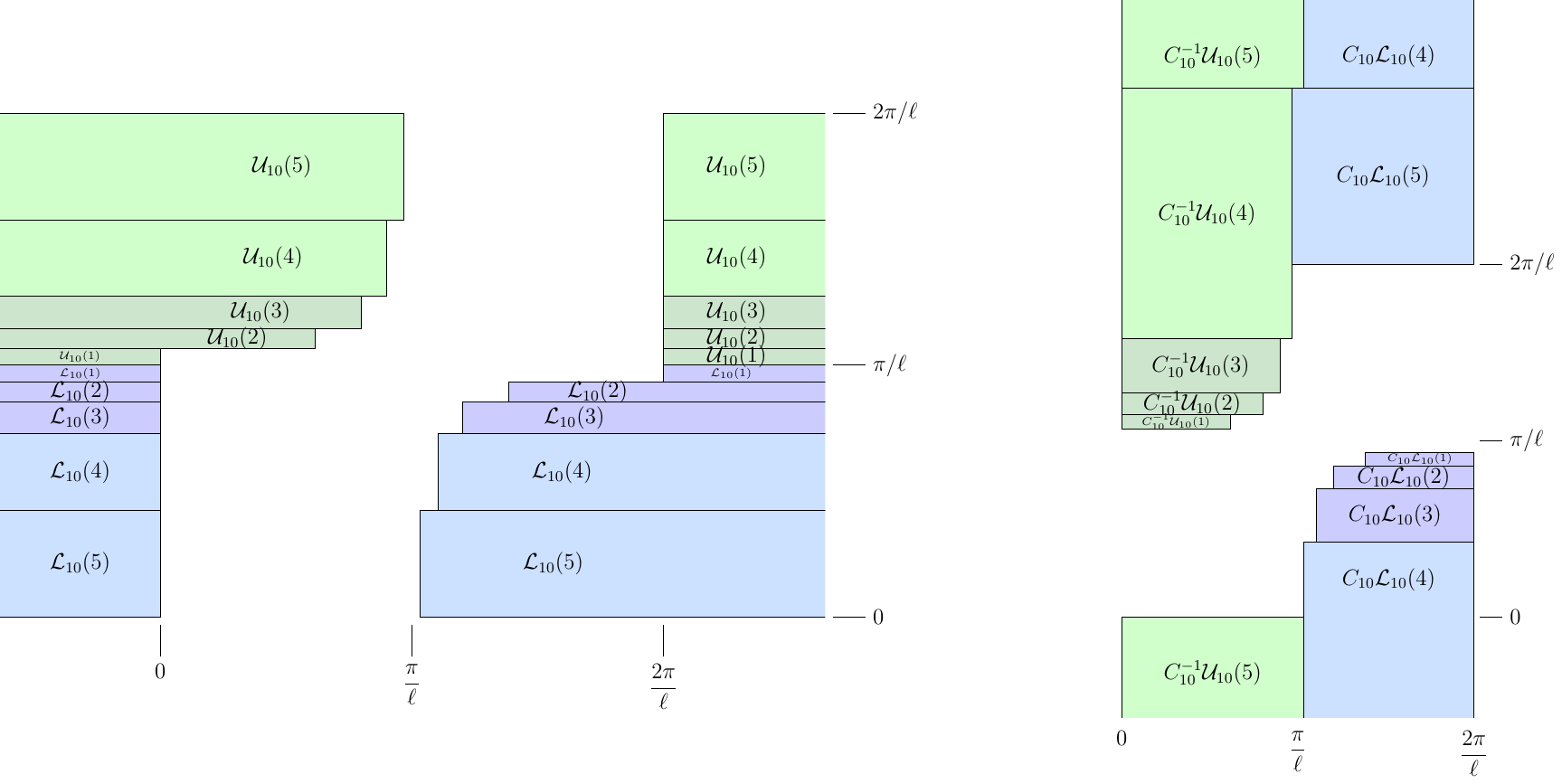}
    \caption{The horizontal strip $\Strip{10}$ and its image under $F_{\A}$}
    \label{fig:strip 10}
\end{figure}
\begin{figure}[ht]
    \includegraphics[width=0.95\textwidth]{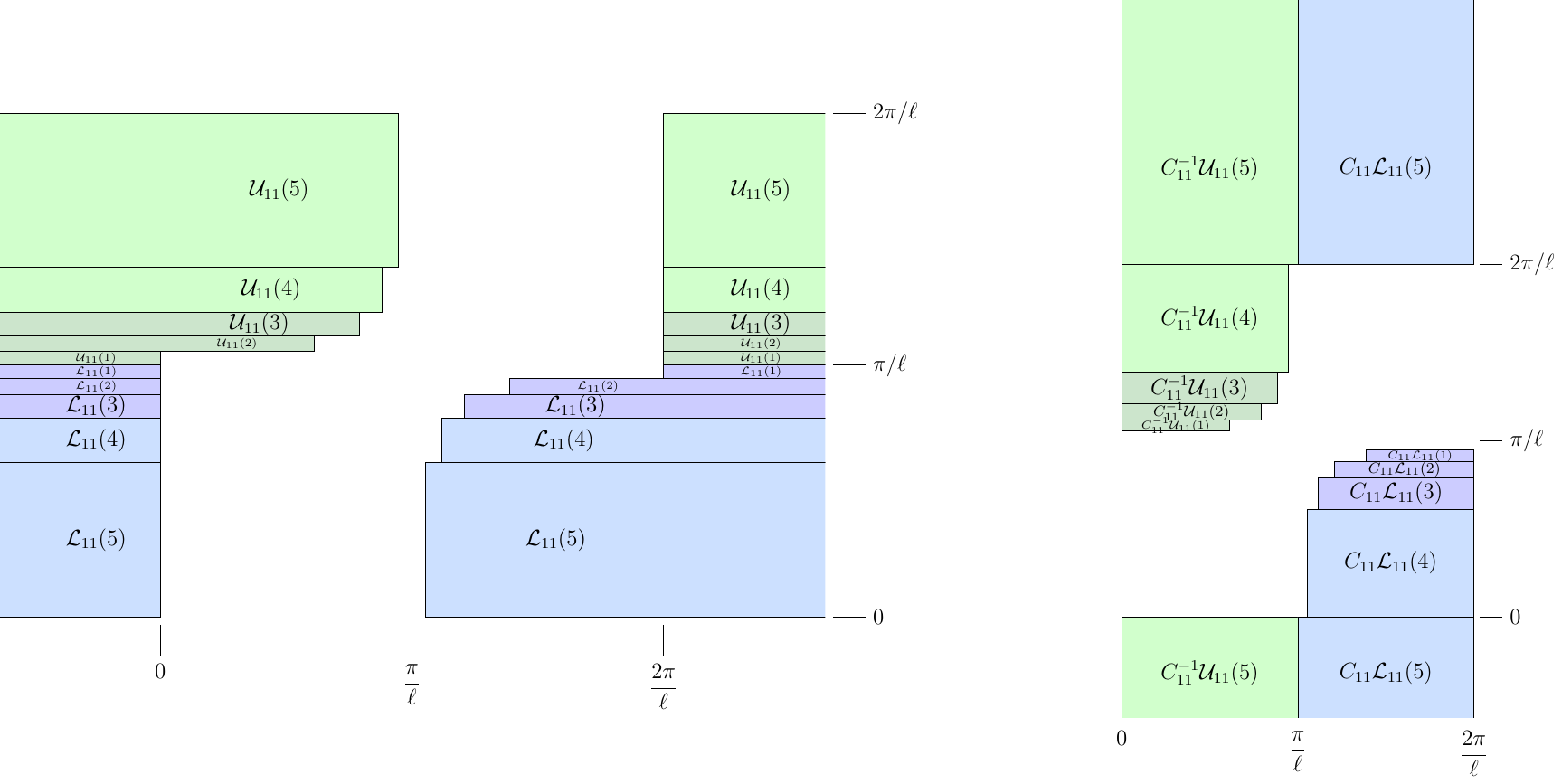}
    \caption{The horizontal strip $\Strip{11}$ and its image under $F_{\A}$}
    \label{fig:strip 11}
\end{figure} 

\begin{thm} \label{thm:bijectivity}
    The map $F_{\A}$ is bijective on the set~$\Omega_{\A}$ defined in~\eqref{omega}.
\end{thm}

\begin{proof}
We claim that $F_{\A}$ maps horizontal strips of~$\Omega_{\A}$ to vertical strips, that is, 
\[ F_{\A}\big(\Omega_{\A} \cap (\S \times [e^{(k-1)(2\pi/\ell)i},e^{k(2\pi/\ell)i}])\big) = \Omega_{\A} \cap ([e^{(k-1)(2\pi/\ell)i},e^{k(2\pi/\ell)i}] \times \S) \]
for $k = 1, ..., \ell$. 
It is sufficient to prove the claim only for $k=1$ by the following reasoning: if $\gamma$ is a side-pairing transformation that glues a side in the $k^\text{th}$ sector \[ \{ z \in \D :  \tfrac{2\pi}\ell (k-1) \le \arg(z) \le \tfrac{2\pi}\ell k \}, \] then $R^{-(k-1)} \circ \gamma \circ R^{k-1}$ is a map that acts in the first sector \[ \{ z \in \D : 0 \le \arg(z) \le \tfrac{2\pi}\ell \}, \] that is, it glues sides of $\Fc_1$ or $\mathcal B(1)$ from page~\pageref{product}.

\medskip
With $k=1$, the claim is that
\[ F_{\A}(\Strip{m}) = \Omega_{\A} \cap ([1,v^2] \times \S) \quad\text{for } m \in \{\infty,\square\}\cup\{2,3,4,...\}. \]

We start with a parabolic horizontal strip
\[
\Strip \infty=([v,1]\times [1,v])\cup ([v^2,v]\times [v,v^2]).
\]
Since $c_\infty(1)=v^2, c_\infty(v)=v, c_\infty^{-1}(v^2)=1, c_\infty^{-1}(v)=v$, we have
\[
F_{\A}(\Strip \infty)=([v,v^2]\times [v^2,v])\cup ([1,v]\times [v,1]),
\]
which by inspection is exactly the expected vertical strip within~$\Omega_{\A}$. That is, $F_{\A}(\Strip \infty)=\Omega_{\A} \cap ([1,v^2]\times \S )$.

The hyperbolic case $\Strip \square$ is similar but consists of four rectangles instead of two.

Next we consider elliptic horizontal strips. When $m=2$ the horizontal strip consists of just one rectangle: $\Strip 2=[v^2,1]\times [1,v^2]$. Since $c_2=c_2^{-1}$,
$c_2(1)=v^2$, and $c_2(v^2)=1$, we have
\[
F_{\A}(\Strip 2)=[1,v^2]\times [v^2,1],
\]
which is the expected vertical strip.

Now let $m\ge 3$. Recall the definitions of $\mathcal L_m(j)$ and $\mathcal U_m(i)$ from \eqref{eq:horizontal strips}. We must calculate the images of $\mathcal L_{m}(j)$ and $\mathcal U_{m}(i)$, which depend on $a = A_{1}\in (1, v^2)$.
For $1 \le j < J+1$ we have
\begin{align*}
\mathcal L_{m}(j) &= [c_m^{j-1}(v^2),1]\times [c_m^j(a),c_m^{j-1}(a)],
\\
F_{\A}(\mathcal L_{m}(j)) &= [c_m^{j}(v^2), v^2] \times [c_m^{j+1}(a), c_m^{j}(a)] \\* &= \mathcal L_{m}(j+1)\cap ([c_m^{J+1}(v^2),v^2]\times \S )\subset\Omega_{\A},
\end{align*}
and for $j=J+1$ we have
\begin{align*}
\mathcal L_{m}(J+1) &= [c_m^{J}(v^2),1]\times [1,c_m^{J}(a)],
\\*
F_{\A}(\mathcal L_{m}(J+1)) &= [c_m^{J+1}(v^2),v^2]\times [v^2, c_m^{J+1}(a)].
\end{align*}
The complex points~$v^2 = e^{(2\pi/\ell)i}$, $c_m^{J+1}(a)$, and $1$ are in counter-clockwise order on the circle, and $[c_m^{J+1}(v^2),v^2]\subset [1,v^2]$.
There are two possible cases:\begin{itemize} \item For odd $m$, $\arg(c_m^{J+1}(a))<0$, then $c_m^{J+1}(a)$ is the end of the cycle. \item For even $m$, $a= c_m^j(v^2)$ for some $j$, and then $c_m^{J+1}(a)=1$.\end{itemize} In either case, we have $F_{\A}(\mathcal L_{m}(J+1)) \subset \Omega_{\A}$ and thus
\[
F_{\A}\left(\bigcup_{j=1}^{J+1}\mathcal L_{m}(j)\right)\subset \Omega_{\A} \quad\text{and}\quad F_{\A}\left(\bigcup_{j=1}^{J+1}\mathcal L_{m}(j)\right)\subset [c_m^{J+1}(v^2),v^2]\times\S.
\]
Similarly, for $1\le i<I+1$ we have  
\begin{align*}
\mathcal U_{m}(i) &= [v^2, c_m^{-i+1}(1)] \times [c_m^{-i+1}(a), c_m^{-i}(a)], 
\\
F_{\A}(\mathcal U_{m}(i)) &= [1, c_m^{-i}(1)]\times [c_m^{-i}(a), c_m^{-i-1}(a)] \\* &= \mathcal U_{m}(i+1)\cap ([1,c_m^{-(I+1)}(1)]\times  \S )\subset\Omega_{\A},
\end{align*}
and for $i=I+1$ we have
\begin{align*}
\mathcal U_{m}(I+1) &= [v^2, c_m^{-I}(1)] \times [c_m^{-I}(a), v^2], 
\\
F_{\A}(\mathcal U_{m}(I+1)) &= [1,c_m^{-(I+1)}(1)]\times [c_m^{-(I+1)}(a), 1]\subset\Omega_{\A}
\end{align*}
since the complex points~$1$, $c_m^{-(I+1)}(a)$, and $v^2$ are in counter-clockwise order, and $[1,c_m^{-(I+1)}(1)]\subset [1,v^2]$.
Thus
\[
F_{\A}\left(\bigcup_{i=1}^{I+1}\mathcal U_{m}(i)\right)\subset\Omega_{\A} \quad\text{and}\quad F_{\A}\left(\bigcup_{i=1}^{I+1}\mathcal U_{m}(i)\right) \subset [1,c_m^{-(I+1)}(1)]\times \S.
\]
Since $c_m^{J+1}(v^2)=c_m^{-(I+1)}(1)$
we conclude that $F_{\A}(\Strip{m})=\Omega_{\A} \cap ([1,v^2]\times  \S )$.
\end{proof}
\subsection{Attractor property}

Having shown (\Cref{thm:bijectivity}) that $F_{\A}$ is bijective on the set~$\Omega_{\A}$, we now want to show that~$\Omega_{\A}$ is the global attractor of $F_{\A}:\S\times\S\setminus\Delta \to \S\times\S\setminus\Delta$. Experimentally, this appears to be true for any choice of elliptic parameters $A_k \in (V_{k-1}, V_{k+1})$, but our existing method of proof (using isometric circles) requires $A_k \in [P_k,Q_k] \subset (V_{k-1}, V_{k+1})$.


\begin{thm}\label{thm:attractor}
If each elliptic partition point satisfies $A_k \in [P_k,Q_k]$, then the set~$\Omega_{\A}$ from~\eqref{omega} is the global attractor of $F_{\A}$ from~\eqref{eq:natural extension}. Moreover, for any initial point $(u_0,w_0)\in \S \times \S \setminus \Delta$, there exists $K > 0$ for which $F_{\A}^K(u_0,w_0)\in \Omega_{\A}$.
\end{thm}

\begin{proof} 
We first show, using isometric circles, that the orbit of any $(u_0,w_0)$ will escape a set $\Phi_{\A}$ defined below, and we then deal with the ``exceptional set'' $\S \times \S \setminus (\Omega_{\A} \cup \Phi_{\A})$.

The set $\Phi_{\A}$  (see~\Cref{fig:Phi}) is the union
\[ \Phi_{\A} := \bigcup_{k=1}^N \Phi_k, \]
where $N$ is the number of sides of the polygon and $\Phi_k$ are defined as follows:
\begin{itemize}
    \item If $V_k$ and $V_{k-1}$ are ideal, we define $\Phi_k = [V_{k-1},V_k] \times [V_{k-1},V_k]$, which is also $[A_{k-1},A_k] \times [A_{k-1},A_k]$.
    \item If $V_k$ is elliptic then $V_{k-1}$ and $V_{k+1}$ are both ideal (so $A_{k-1}=V_{k-1}$ and $A_{k+1}=V_{k+1}$), and we define $\Phi_k = [V_{k-1},Q_k] \times [A_{k-1},A_k]$ and $\Phi_{k+1} = [P_k,V_{k+1}] \times [A_{k},A_{k+1}]$.
\end{itemize}

\begin{figure}[htb]
\includegraphics[width=0.75\textwidth]{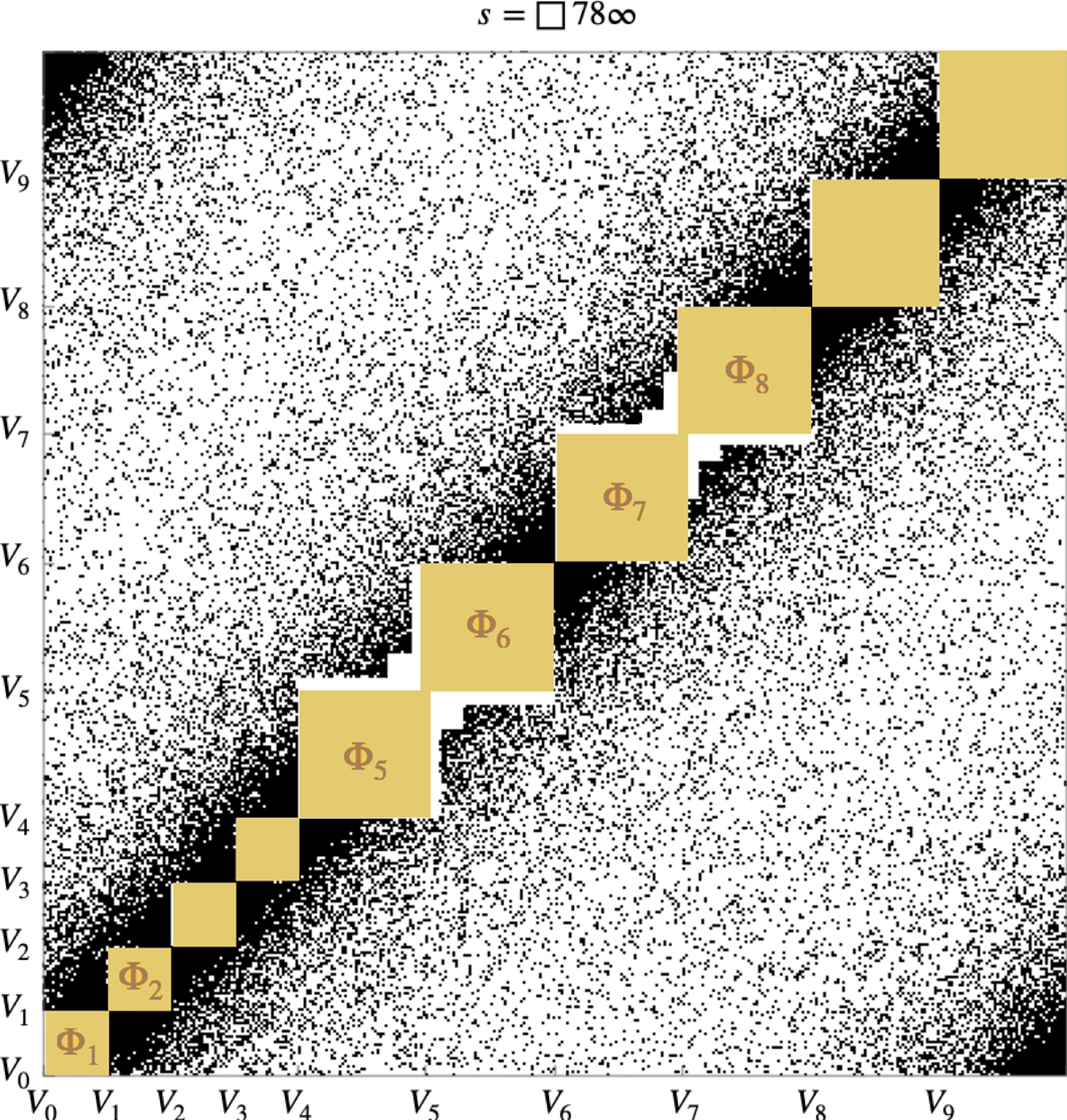}
\caption{By isometric circle argument, orbits must escape the colored set}
\label{fig:Phi}
\end{figure}

Each iterate $(u_n,w_n) := F_{\A}^n(u_0,w_0)$ is acted on by a M\"obius transformation, which we call $\g_k$. By Property~(\ref{def:isom}) of \Cref{def:canonical}, the isometric circle of $\g_k$ contains a side of the canonical polygon, and $\g_k$ is the side-pairing transformation.
From this we can conclude that there exists a positive integer $n$, dependent on $u_0$ and $w_0$, such that the point $(u_n,w_n)$ escapes the set $\Phi_{\A}$ because each M\"obius transformation $\gamma_{k}$ is expanding in the interior of its isometric circle, including the intervals that define each $\Phi_k$. Thus the distance between $u_n$ and $w_n$ must grow sufficiently for the points to eventually belong to different isometric circles, hence there exists $n\ge 0$ such that $(u_n,w_n)\notin\Phi_{\A}$. 

\medskip
In order to prove the attracting property of the set~$\Omega_{\A}$, we need only to consider points $(u,w)\notin (\Omega_{\A} \cup \Phi_{\A})$ and show that they will enter~$\Omega_{\A}$ after finitely many iterations of $F_{\A}$. Notice that a set $[A_{k-1},A_{k+1}]\times [A_{k-1},A_{k+1}]\setminus (\Omega_{\A} \cup \Phi_{\A})$ is non-empty only when $V_k$ is an elliptic point of order greater than $2$, so we investigate such a situation. These ``exceptional sets'' are the white areas of \Cref{fig:Phi}.

We may assume that we are dealing with the standard-position $\mathcal B(1)$, that is, we have $w \in [1,v^2]$.
The exceptional set for the elliptic vertex $V_1$ are shown as multicolored rectangles in \Cref{fig:exceptional}.

\begin{figure}[htb]
\begin{tikzpicture}
    \draw (3,3) node {\includegraphics[width=6cm]{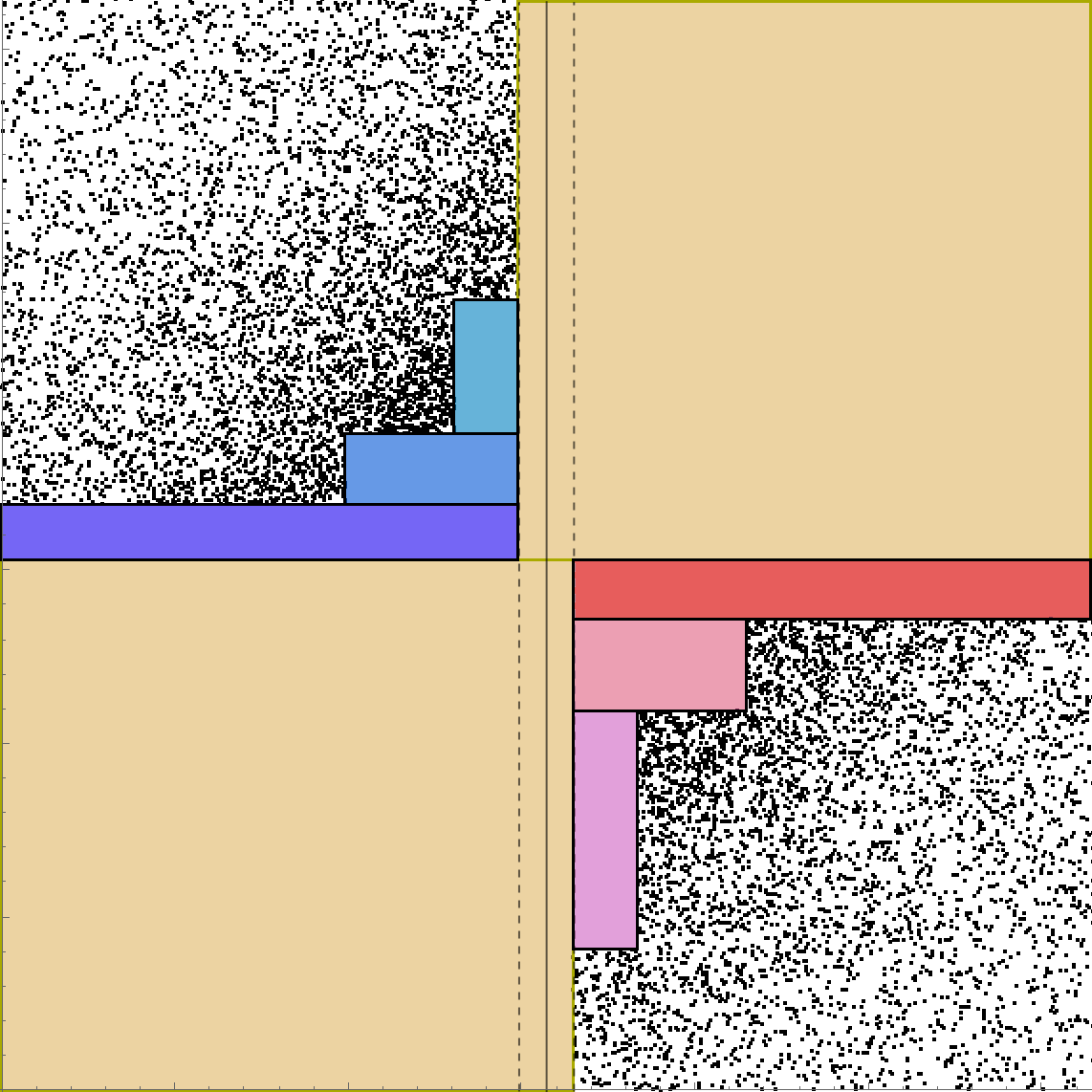}};
    \draw (0,6) node [left] {$v^2$} -- (0,2.9) node [left] {$a$} -- (0,0) node [left] {$1$} node [below] {$1$} -- (2.85,0) node [below] {$p$} -- (3.15,0) node [below] {$q$} -- (6,0) node [below] {$v^2$};
    \draw [yellow!33!black] (1.5,1.5) node {$\Phi_1$} (4.5,4.75) node {$\Phi_2$};
    \draw [thick,black] (2.4,3.05) -- (1.5,2.5) node [left,blue!50!gray] {$\widehat{\mathcal U}_8(1)$};
    \draw [thick,black] (3.5,2.8) -- (4.5,3.5) node [right,red!67!gray] {$\widehat{\mathcal L}_8(1)$};

    \begin{scope}[xshift=8cm]
    \draw (3,3) node {\includegraphics[width=6cm]{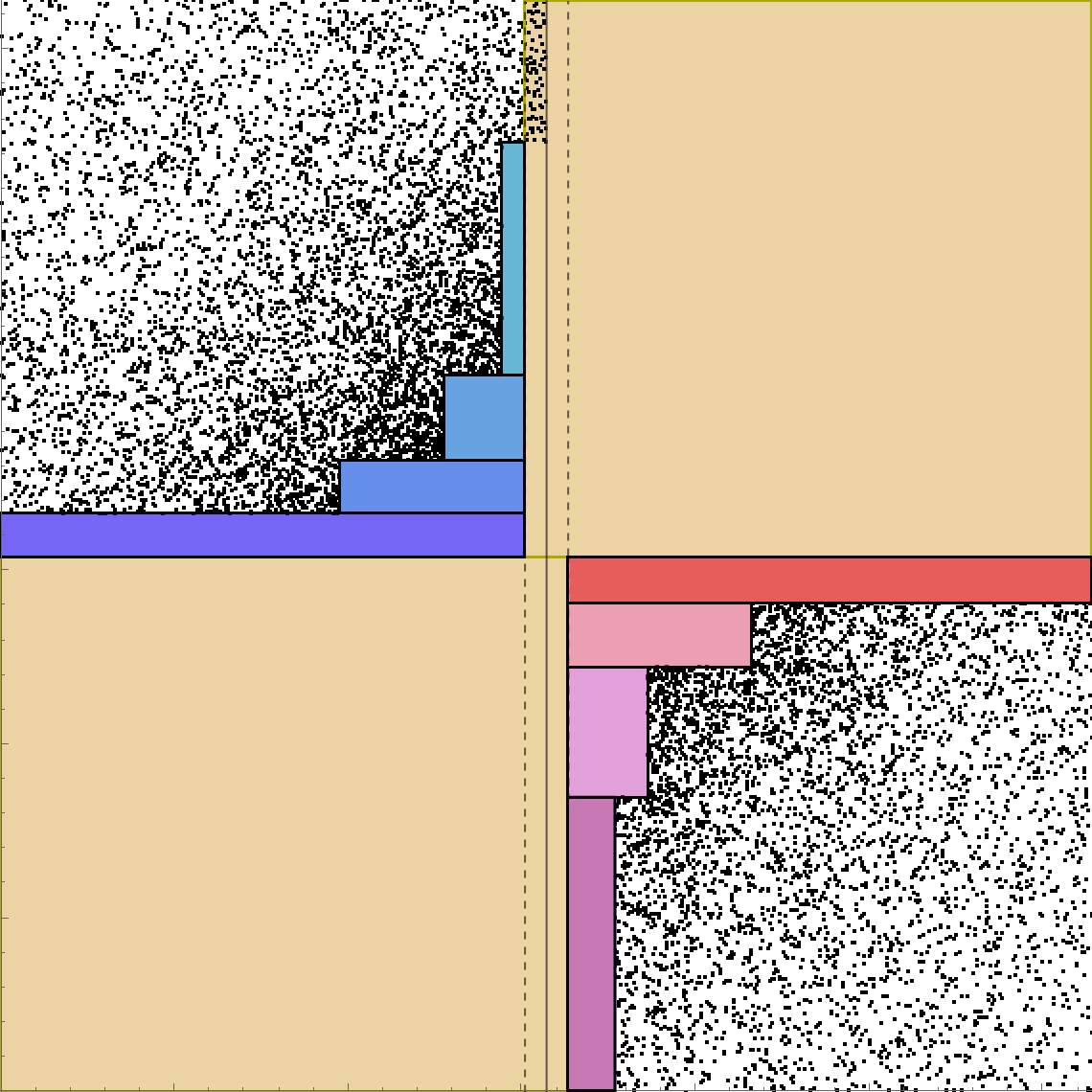}};
    \draw (0,6) node [left] {$v^2$} -- (0,2.9) node [left] {$a$} -- (0,0) node [left] {$1$} node [below] {$1$} -- (2.85,0) node [below] {$p$} -- (3.15,0) node [below] {$q$} -- (6,0) node [below] {$v^2$};
    \draw [yellow!33!black] (1.5,1.5) node {$\Phi_1$} (4.5,4.75) node {$\Phi_2$};
    \draw [thick,black] (2.4,3.05) -- (1.5,2.5) node [left,blue!50!gray] {$\widehat{\mathcal U}_9(1)$};
    \draw [thick,black] (3.5,2.8) -- (4.5,3.5) node [right,red!67!gray] {$\widehat{\mathcal L}_9(1)$};
    \draw [thick,black] (3.25,0.7) -- (2.5,0.5) node [left,magenta!50!black] {$\widehat{\mathcal L}_9(4)\!$};
    \end{scope}
\end{tikzpicture}
\caption{Rectangles of the exceptional set for $m=8$ (left) and $m = 9$ (right)}
\label{fig:exceptional}
\end{figure}

\medskip\noindent\textbf{Case $\boldsymbol m$ even.} In this case, the exceptional lower rectangles are
\begin{align}\label{eq:lower}
    \widehat{\mathcal L}_m(j) &= [q,c_m^{j-1}(v^2)]\times [c_m^j(a),c_m^{j-1}(a)], \quad 1\le j \le J
\end{align}
and the upper rectangles are
\begin{align}\label{eq:upper}
    \widehat{\mathcal U}_m(i) &=[c_m^{-i+1}(1),p]\times [c_m^{-i+1}(a), c_m^{-i}(a)], \quad 1\le i \le I 
\end{align}
(There are no exceptional rectangles corresponding to ${\mathcal U}_m(I+1)$ and ${\mathcal L}_m(J+1)$.)

It is enough to track the images of $\widehat{\mathcal L}_m(1)$ and $\widehat{\mathcal U}_m(1)$ only. Indeed, if $1\le j<J$,
\begin{equation}\label{eq:subset}
    F_{\A}(\widehat{\mathcal L}_m(j))  =c_m(\widehat{\mathcal L}_m(j))=[p,c_m^j(v^2)]\times [c_m^{j+1}(a),c_m^j(a)]\supset \widehat{\mathcal L}_m(j+1).
\end{equation}
This shows that $\widehat{\mathcal L}_m(j)\subset F_{\A}^{j-1}(\widehat{\mathcal L}_m(1))$ for $2\le j\le J$.
Similarly, one can show that $\widehat{\mathcal U}_m(i)\subset F_{\A}^{i-1}(\widehat{\mathcal U}_m(1))$ for $2\le i\le I$. This allows us to focus on  the exceptional sets  $\widehat{\mathcal L}_m(1)$ and $\widehat{\mathcal U}_m(1)$.

If $(u,w)\in \widehat{\mathcal L}_m(1)$ (red rectangle in \Cref{fig:exceptional-images} left), then
\begin{align*}
F_{\A}^{J}(\widehat{\mathcal L}_m(1)) 
&= [c_m^{J}(q),c_m^{J}(v^2)]\times [c_m^{J+1}(a),c_m^{J}(a)] \\*
&= [c_m^{J}(q),c_m^{J}(v^2)]\times [c_m^{J+1}(a),1]\cup [c_m^{J}(q),c_m^{J}(v^2)]\times [1,c_m^{J}(a)].
\end{align*}
The first rectangular piece belongs to the upper part of~$\Omega_{\A}$,  since $\arg(c_m^{J+1}(a))\ge\arg(v^2)=\frac{2\pi}{\ell}$ (see also the proof of \Cref{thm:bijectivity}). The image of the second piece (green rectangle) is
\begin{align*}
F_{\A}\big([c_m^{J}(q),c_m^{J}(v^2)]\times [1,c_m^{J}(a)]\big) 
&=[c_m^{J+1}(q),c_m^{J+1}(v^2)]\times [c_m(1),c_m^{J+2}(a)] \\
&=[1,p]\times [v^2,c_m^{J+2}(a)] \\
&\subset [1,p] \times [v^2, 1] \;\;\subset\;\; \Omega_{\A}.
\end{align*}
One can track $\widehat{\mathcal U}_m(1)$, similarly, to prove that all its points will enter~$\Omega_{\A}$, after finitely many iterations.

\begin{figure}[htb]
\def\s{0.9}
\begin{tikzpicture}[scale=0.9]
    \draw (3,5.95) node {\includegraphics[width=5.4cm]{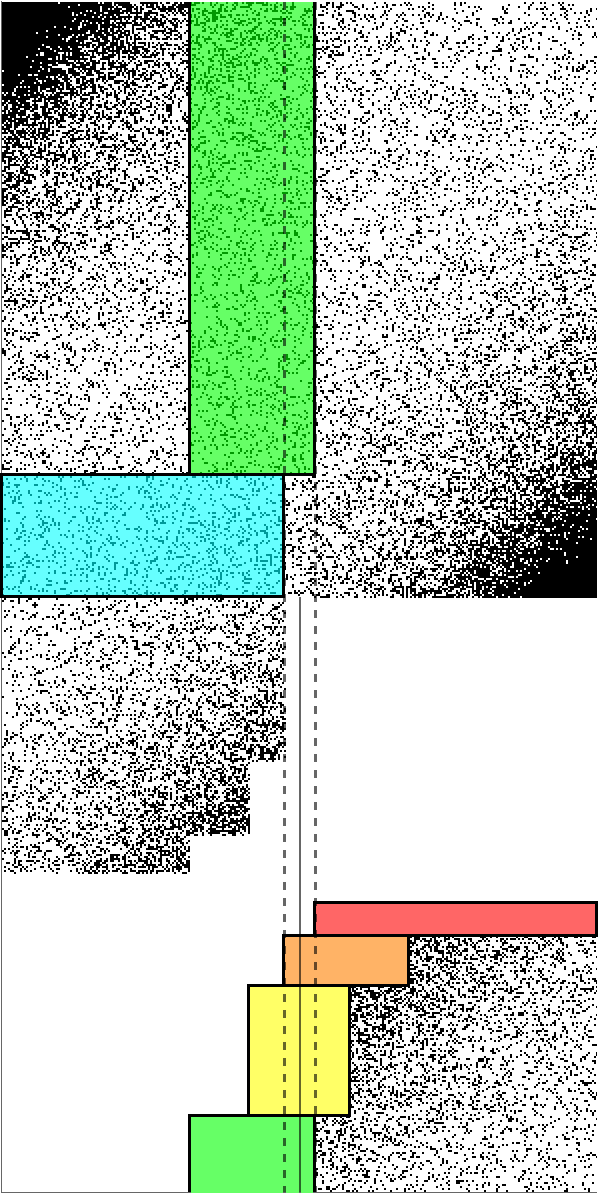}}; 
    \draw (0,6) node [left] {$v^2$} (0,0) node [left] {$1$} node [below] {$1$} (2.85,0) node [below] {$p$} (3.15,0) node [below] {$q$} (6,0) node [below] {$v^2$};
    \draw [yellow!50!black,dashed] (3.15,2.89) -- (0,2.89) node [left,black] {$a$};
    \draw [thick,black] (3.5,2.78) -- (4.5,3.5) node [above,red!67!gray] {\;\;\;$\widehat{\mathcal L}_8(1)$};

    \begin{scope}[xshift=8cm]
    \draw (3,5.95) node {\includegraphics[width=5.4cm]{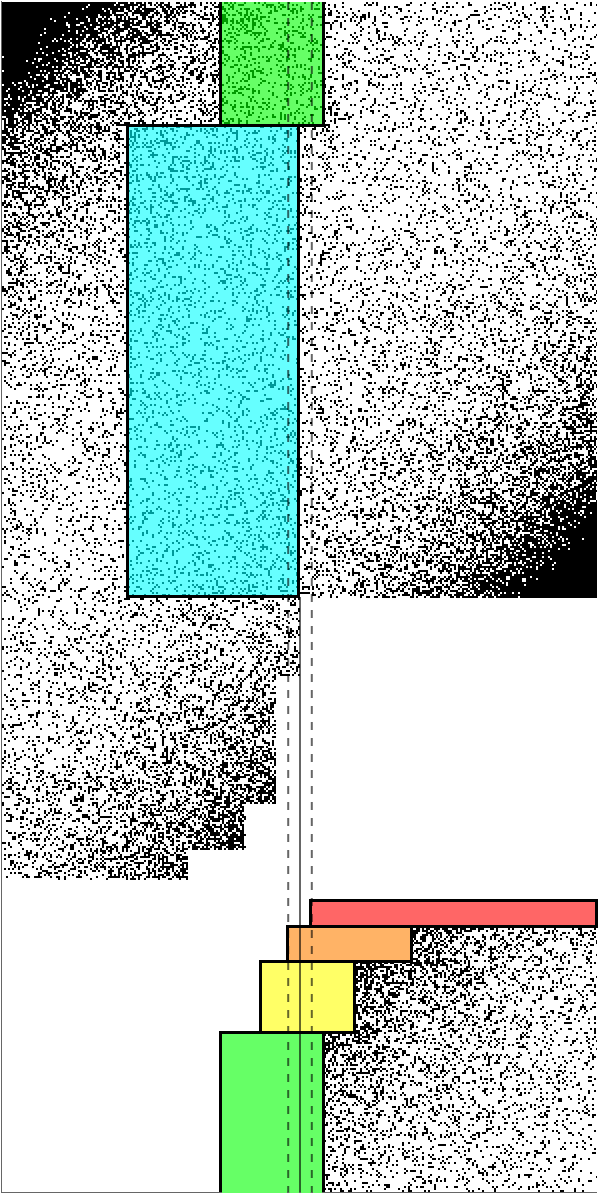}}; 
    \draw (0,6) node [left] {$v^2$} (0,0) node [left] {$1$} node [below] {$1$} (2.85,0) node [below] {$p$} (3.15,0) node [below] {$q$} (6,0) node [below] {$v^2$};
    \draw [yellow!50!black,dashed] (3.11,2.91) -- (0,2.91) node [left,black] {$a$};
    \draw [thick,black] (3.5,2.78) -- (4.5,3.5) node [above,red!67!gray] {\;\;\;$\widehat{\mathcal L}_9(1)$};
    \end{scope}
\end{tikzpicture}
\caption{Images of exceptional rectangle $\widehat{\mathcal L}_8(1)$ (left) and $\widehat{\mathcal L}_9(1)$ (right). The cyan rectangle is the image of only the part of the green rectangle that is outside $\Omega_\A$}
\label{fig:exceptional-images}
\end{figure}

\medskip

\noindent\textbf{Case $\boldsymbol m$ odd.} In this case, in addition to the exceptional lower and upper rectangles described in \eqref{eq:lower} and \eqref{eq:upper}
we have an exceptional lower rectangle
\[
\widehat{\mathcal L}_m(J+1) = [q,c_m^{J}(v^2)]\times [1,c_m^{J}(a)] \; \text{ if }\; a\in [p, v]
\]
or an exceptional upper rectangle
\[
\widehat{\mathcal U}_m(I+1) = [c_m^{-I}(1),p]\times [c_m^{-I}(a), v^2] \; \text{ if } \; a\in (v,q].
\]
In what follows, we treat the case $a\in [p,v]$ with the additional exceptional rectangle $\widehat{\mathcal L}_m(J+1)$. 
The inclusions $\widehat{\mathcal L}_m(j)\subset F_{\A}^{j-1}(\widehat{\mathcal L}_m(1))$ for $2\le j\le J$ and $\widehat{\mathcal U}_m(i)\subset F_{\A}^{i-1}(\widehat{\mathcal U}_m(1))$ for $2\le i\le I$ remain valid, as in the case $m$ even. Also 
\[F_{\A}(\widehat{\mathcal L}_m(J))=c_m(\widehat{\mathcal L}_m(J))=[p,c_m^{J}(v^2)]\times [1,c_m^{J}(a)]\supset \widehat{\mathcal L}_m(J+1)\]
since $c_m(q) = p$ and  the points $1, p, v, q, v^2$ are in counter-clockwise order.
Thus we only need to analyze the iterates of points in the exceptional sets  $\widehat{\mathcal L}_m(1)$ and $\widehat{\mathcal U}_m(1)$.

\smallskip

If $(u,w)\in \widehat{\mathcal L}_m(1)$ (red rectangle in \Cref{fig:exceptional-images} right), then 
\begin{align*}
F_{\A}^{J+1}(\widehat{\mathcal L}_m(1)) &=[c_m^{J+1}(q),v]\times [v^2,1] \;(\text{green rectangle}) \\
&\subset [1,v] \times [v^2,1] \;\subset\; \Omega_{\A}
\end{align*}
because  $c_m^{J+1}(q) = c_m^{J}(p) \in [1,v]$ (see also \Cref{fig:flower}). 

One can track $\widehat{\mathcal U}_m(1)$, similarly, to prove that all its points will enter~$\Omega_{\A}$, after finitely many iterations.
\end{proof}

The proof of \Cref{thm:attractor} uses isometric circles to show that the orbit escapes $\Phi_{\A}$, and for this reason we require $A_k \in [P_k,Q_k]$ in the statement of the theorem. Most likely this restriction in unnecessary:
\begin{conjecture*}\label{con1}
If each elliptic partition point satisfies $A_k \in (V_{k-1},V_{k+1})$, then the set~$\Omega_{\A}$ from~\eqref{omega} is the global attractor for $F_{\A}$.
\end{conjecture*}

With the canonical case complete, we now prove the \hyperlink{main}{main theorem}.

\section{Proof of Main Theorem}\label{sec:main proof}

Let~$\G$ be any Fuchsian group of the first kind with signature $(g;m_1,\dots,m_r;t\ge 1)$.
Let~$\Fc^*$ be the canonical marked quasi-ideal polygon with this signature and $\G^*$ the group generated by side-pairing transformations of~$\Fc^*$. 
By \Cref{thm:Deform} there exists a Fenchel--Nielsen map $h:\bar\D\to\bar\D$ such that $\G=h\circ\G^*\circ h^{-1}$.

For the canonical case, let $\A^* = \{A^*_1,\dots,A^*_N\}$ be a partition of~$\S$ with $A^*_k\in [P^*_k,Q^*_k]$ if $A^*_k$ is an elliptic partition point and $A^*_k=V^*_k$ if $A^*_k$ is an ideal partition point,
\[
f_{\A^*}(x)=\gamma^*_k(x)\text{ for }x\in [A^*_{k-1},A^*_k)
\] be the associated boundary map, and 
\[F_{\A^*}(u,w)=(\gamma^*_k(u),\gamma^*_k(w))\text{ for }w\in [A^*_{k-1},A^*_k)
\] be its natural extension.

These data are transferred to the data for~$\G$ by the Fenchel--Nielsen map:
\begin{itemize}
    \item By~\eqref{eq:geometric}, the maps $\gamma_k := h\circ \gamma^*_k\circ h^{-1}$ generate~$\G$.
    \item $\A=h(\A^*) = \{A_1,\dots,A_N\}$ is a partition of~$\S$. Ideal $A^*_k$ are mapped to ideal $A_k$ by construction, and if $A^*_k\in [P^*_k,Q^*_k]$ is an elliptic partition point, then $A_k$ belongs to a non-degenerate closed interval $[P_k,Q_k] := [h(P^*_k),h(Q^*_k)]$. Note that $\arg(P_k)<\arg(Q_k)$ since $h$ preserves the order of points on~$\S$. Notice that $P_k$ and $Q_k$ are not necessarily the second endpoints of the geodesic rays $(V_{k-1},V_k)$ and $(V_{k+1},V_k)$.
    \item $f_{\A} (x)=\gamma_k(x)$ for $x\in [A_{k-1},A_{k})$ is the boundary map acting by generators of the group~$\G$.
    \item $F_{\A}(u,w) = (\gamma_k(u),\gamma_k(w))\text{ for }w\in [A_{k-1},A_k)$ is its natural extension map.
\end{itemize}
Then
\[
f_{\A} = h\circ f_{\A^*}\circ h^{-1}
\]
and
\[
F_{\A}=(h\times h)\circ F_{\A^*} \circ(h^{-1}\times h^{-1}).
\]
Note that if we start with any partition $\A$ for~$\G$ with elliptic partition points $A_k\in [P_k,Q_k]$, then
\begin{itemize}
    \item $h^{-1}(\A) = \A^*$ will be a partition for $\G^*$ with elliptic partition points $A^*_k\in [P^*_k,Q^*_k]$, 
    \item $f_{\A^*}=h^{-1}\circ f_{\A} \circ h$ will be a boundary map for $\G^*$, and
    \item $F_{\A^*}=(h^{-1}\times h^{-1})\circ F_{\A} \circ(h\times h )$ will be the natural extension map for $f_{\A^*}$.
\end{itemize}

Since the map $f_{\mathcal P^*}$ is Markov by  \Cref{prop:markov}, we know that $f_{h(\mathcal P^*)}$ is also Markov, which proves item~\ref{main-i} of the main theorem (we could use $\mathcal Q^*$ or $\mathcal M^*$ as well).
 
By \Cref{thm:attractor}, $F_{\A^*}$ has global attractor~$\Omega_{\A^*}$ with finite rectangular structure.
Then $(h\times h)(\Omega_{\A^*})$ is the global attractor (and domain of bijectivity) for $F_{\A}$, and since the map $h\times h$ preserves horizontal and vertical lines, this domain has finite rectangular structure as well, proving~\ref{main-ii} and~\ref{main-iii}. \qed

\section{Application to coding of geodesic flow}
\label{sec:remarks}The finite rectangular structure of the attractor for the reduction map, in addition to Don Zagier's inversion problem for modular forms \cite{Zagier-2023} mentioned in the introduction, can be also used for coding geodesics on $\G\backslash \D$. For previous related work, see the results of \cite{KU-ab-applications} for the modular group and \cite{AK-revisited} for surface groups.


A geodesic in $\D$ from $u$ to $w$ ($u,w\in\S$) is called {\em reduced} if $(u,w)\in\Omega_\A$. According to the main theorem, every geodesic can be reduced in finitely many steps. Let us start with a reduced geodesic from $u_0$ to $w_0$. We associate to $w_0$ a symbolic sequence
\[
[w_0]=[n_0,n_1,n_2,\dots]
\]
in the finite alphabet $\{1,2,3,\dots, 4g+2r+2(t-1)\}$ that corresponds to the sides of~$\Fc$,
where for $k\ge 0$ if $f_\A^k(w_0)\in [A_i,A_{i+1})$, then~$n_k$ is the side of~$\Fc$ that is paired with side $i$.
By successive application of~$F_\A$ we obtain a sequence of reduced geodesics from $u_k$ to $w_k$, where 
\[
[w_k]=[n_k,n_{k+1},\dots].
\]
From the bijectivity of $F_\A$ on $\Omega_\A$ we can continue the sequence to the past, thus associating to a given geodesic on the orbifold a bi-infinite symbolic sequence
\[
[\dots n_{-2},n_{-1},n_0,n_1,n_2,\dots].
\]
The left shift of the sequence corresponds to an application of the map $F_\A$ to the corresponding geodesic. Using the Markov partitions $\mathcal P, \mathcal M, \mathcal Q$, one can obtain a Markov symbolic coding.
The parameterization of the cross section by the attractor will allow us to represent the geodesic flow on $\G\backslash \D$ as a special flow over a symbolic dynamical system. This application is a subject of current work by the authors and will appear elsewhere.

\appendix
\section{Explicit construction of the canonical polygon}\label{sec:explicit}

Given signature $(g;m_1,\dots,m_r;t\ge 1)$, for each cyclic subgroup we construct a canonical fundamental polygon. Recall that $\ell=g+r+t-1$ is the number of building blocks (number of sectors) in the fundamental polygon. Also recall equation \eqref{string}, which defines the string
\begin{equation*}
    s \;\;=\;\;
    [\STRING]
    \;\;=\;\;
    \{\underbrace{\square,\dots,\square}_{\text{\scriptsize$g$}}, m_1,\dots, m_r, \underbrace{\infty,\dots,\infty}_{\text{\scriptsize$t-1$}}\}
\end{equation*}
corresponding to the signature $(g; m_1,m_2,\dots, m_r;t\ge 1)$.
We now define the infinite-area polygons
\providecommand\setbuilder[2]{\ensuremath{\left\{#1:#2\right\}}}
\begin{align*}
    \Fc_m &:= \setbuilder{z\in\D}{ \abs{z-\sec(\theta)e^{\theta i}} \ge \tan(\theta) \text{ and}\abs{z-\sec(\theta)e^{(2\pi/\ell-\theta) i}} \ge \tan(\theta) }, \\*&\qquad\qquad\text{where }\theta = \arctan\big(\tfrac{\sin(\pi/\ell)}{\cos(\pi/\ell)+\cos(\pi/m)}\big).
\end{align*}
In the case $m=2$ the two Euclidean circles bounding $\Fc_m$ coincide, and in the degenerate case $m=\ell=2$ it becomes a line: $\Fc_2 = \setbuilder{z\in\D}{\Im(z)\le0}$ when $\ell=2$. In the limiting case $m = \infty$ this becomes
\begin{align*}
    \Fc_\infty &:= \setbuilder{z\in\D}{ \abs{z-\sec(\theta)e^{\theta i}} \ge \tan(\theta) \text{ and}\abs{z-\sec(\theta)e^{(2\pi/\ell-\theta) i}} \ge \tan(\theta) }, \\*&\qquad\qquad\text{where }\theta = \arctan\big(\tfrac{\sin(\pi/\ell)}{\cos(\pi/\ell)+1}\big).
\end{align*}
When $g > 0$ we also use
\begin{align*}
    \Fc_\square &:= \setbuilder{z\in\D}{ \abs{z-\sec(\theta)e^{(2k+1)\theta i}} \ge \tan(\theta), \; k=0,1,2,3 }, \quad\text{where }\theta=\tfrac{\pi}{4\ell}.
\end{align*}

\begin{figure}[hbt]
    \includegraphics[width=0.67\textwidth]{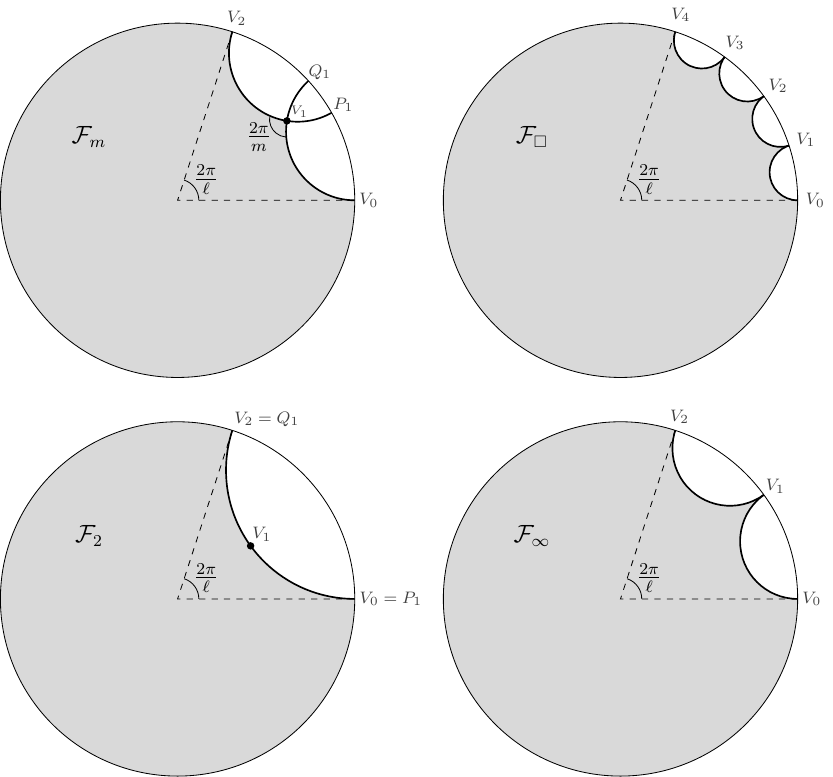}
    \caption{Standard-position fundamental polygons}
    \label{fig:our free combinations}
\end{figure}

The markings on these polygons are as follows:
\begin{itemize}[leftmargin=1cm]
    \item In $\Fc_m$ the non-ideal vertex is \[ V_1 = {\cos(\tfrac{\ell+m}{2\ell m}\pi)}{\sec(\tfrac{\ell-m}{2\ell m}\pi)}e^{(\pi/\ell)i}, \] with the geodesic arcs $(1,V_1)$ and $(V_1,e^{2\pi i/\ell})$ meeting there with angle $\frac{2\pi}{m}$. The former arc is glued to the latter by the map $c_m : \bar\D \to \bar\D$ that rotates by angle $\frac{2\pi}m$ around~$V_1$. As a matrix acting on the unit disk, $c_m$ is given explicitly by 
    \[ c_m = \Matrix{ 1 & -V_1 \\ -\overline{V_1} & 1 }^{-1} \!\Matrix{ e^{-\pi i/m} & 0 \\ 0 & e^{\pi i/m} }\! \Matrix{ 1 & -V_1 \\ -\overline{V_1} & 1 }, \]
    or, equivalently, 
    \[ c_m(z) = \frac{(1+\cos(\tfrac\pi m)e^{\pi i/\ell})z - (\cos(\tfrac\pi m) + \cos(\tfrac\pi \ell))e^{\pi i/\ell}}{\big((\cos(\tfrac\pi m) + \cos(\tfrac\pi \ell))e^{-\pi i/\ell}\big) z - (1+\cos(\tfrac\pi m)e^{-\pi i/\ell})} \] 
    Note that $\ang{c_m}$ is a cyclic group  of order~$m$.
    \item For $m=\infty$, the point $V_1 = e^{\pi i/\ell}$ is ideal and the side~$(1,V_1)$ is glued to~$(V_1,V_2)$ by \[ c_\infty(z) = \frac{2e^{2\pi i/\ell}z - (e^{2\pi i/\ell} + e^{3\pi i/\ell})}{(e^{\pi i/\ell}+1)z-2e^{\pi i/\ell}}. \]
    \item For $\Fc_\square$, denote $V_k = e^{k \pi i/(2\ell)}$ so that the vertices of $\Fc_\square$ are $V_{0}$, $V_{1}$, $V_{2}$, $V_{3}$, $V_{4}$, all of which are ideal. The side $V_{0}V_{1}$ is glued to $V_{3}V_{2}$ (note the orientations) by
    \[ a_1(z) = \frac{ -e^{5\pi i/(4\ell)}z + \cos(\tfrac\pi{4\ell})e^{3\pi i/(2\ell)} }{ -\cos(\tfrac\pi{4\ell})z + e^{\pi i/(4\ell)} },
    \] 
    and the side $V_{3}V_{4}$ is glued to $V_{2}V_{1}$ by \[ b_1(z) = e^{\pi i/(2\ell)} \cdot a_1^{-1}(e^{-\pi i/(2\ell)} z) 
    \] 
    so that the glued geodesics are isometric circles. Note that the transformation $b_1^{-1}a_1^{-1}b_1a_1$ maps $V_{0}$ to $V_{4}$.
\end{itemize}

Some examples of this construction are shown in \Cref{fig:polygon examples} (note that the polygon \Cref{fig:polygon examples}(a) in the disk model is exactly the right part of \Cref{fig:two modular polygons} in the half-plane model).

\begin{figure}[hbt]
    \includegraphics[width=0.9\textwidth]{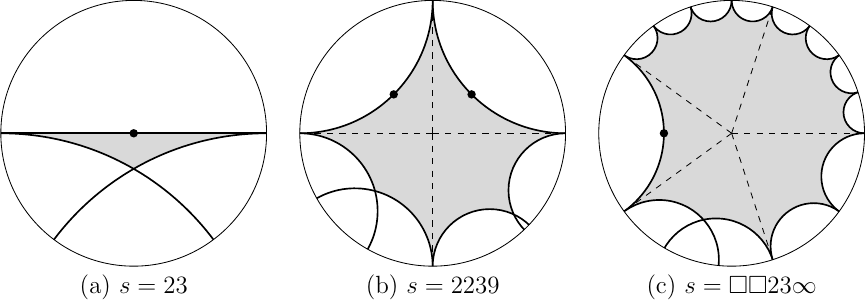}
    \caption{Examples of quasi-ideal polygons}
    \label{fig:polygon examples}
\end{figure}

%

\bibliographystyle{plain}
\bibliography{references}
\end{document}